\RequirePackage{fix-cm} 
\documentclass[reqno]{amsart}

\pdfoutput=1

\usepackage{bbold}			
\usepackage{enumitem}	
	
\usepackage{amsmath}	
\usepackage{amscd}
\usepackage{amsthm}
\usepackage{changepage} 
\usepackage{tikz-cd}	
\usepackage{stmaryrd}	
\usepackage{MnSymbol}	
\usepackage{slashed} 	
	\newcommand{\Dirac}[1]{\slashed{\partial}_{#1}}
\usepackage{hyperref}	
\usepackage[shortcuts]{extdash} 
\usepackage{cleveref}	
	\crefname{subsection}{subsection}{subsections} 
	\crefdefaultlabelformat{#2\textup{#1}#3} 
\crefformat{diagram}{Diagram~(#2#1#3)} 
\Crefformat{Diagram}{Diagram~(#2#1#3)}
\crefformat{property}{Property~#2#1#3} 
\Crefformat{Property}{Property~#2#1#3}
\newenvironment{Claimproof}
{\begin{adjustwidth}{0.25cm}{}
		\begin{proof}[Proof of claim.]
		}
		{ 
		\end{proof} 
	\end{adjustwidth}
}

\newtheoremstyle{theorems}
{1\baselineskip}
{0.75\baselineskip}
{\itshape}
{}
{\bfseries}
{.}
{3pt}
{}

\newtheoremstyle{lemmas}
{1\baselineskip}
{0.75\baselineskip}
{\itshape}
{}
{\bfseries}
{.}
{3pt}
{}

\newtheoremstyle{rest}
{1\baselineskip}
{1\baselineskip}
{}
{}
{\itshape} 
{\emph{:}}
{5pt}
{}

\newtheoremstyle{claim}
{1\baselineskip}
{0.5\baselineskip}
{\itshape}
{}
{}
{.}
{5pt}
{}

\theoremstyle{claim}
	\newtheorem{claim}{Claim}

\theoremstyle{theorems}
	\newtheorem{thm}{Theorem}[section]
	\newtheorem{prop}[thm]{Proposition}

\theoremstyle{lemmas}
	\newtheorem{lem}[thm]{Lemma}
	\newtheorem{cor}[thm]{Corollary}
	\newtheorem*{lem*}{Lemma} 

\theoremstyle{rest}
	\newtheorem{rmk}{Remark}
	\newtheorem{defn}{Definition}
	\newtheorem{example}{Example}
	\newtheorem*{fact*}{Fact}	
	\newtheorem*{recall*}{Recall}
	\newtheorem*{notation}{Notation}
	
\makeatletter
\def\namedlabel#1#2{\begingroup#2%
	\def\@currentlabel{#2}%
	\phantomsection\label{#1}\endgroup
}
\makeatother

\DeclareMathOperator{\End}{End}
\DeclareMathOperator{\Kth}{K}
\DeclareMathOperator{\dom}{\mathrm{dom}}

\newcommand{\B}{\mathcal{B}}
 
\newcommand{\D}{\mathcal{D}}

\newcommand{\M}[1]{{\mathsf{M}_{#1}}} 
\newcommand{\NN}{\mathbb{N}}
\newcommand{\RR}{\mathbb{R}}
\renewcommand{\SS}{\mathbb{S}}
\newcommand{\Schw}{\mathcal{S}}
\newcommand{\ZZ}{\mathbb{Z}}

\newcommand{\abs}[1]{\left| #1 \right|}
\newcommand{\norm}[1]{\left\|#1\right\|} 
\newcommand{\supnorm}[1]{\left\| #1 \right\|_{\infty}\!}
\newcommand{\inner}[2]{\left\langle{ #1\,\vert\, #2}\right\rangle}
\newcommand{\biginner}[2]{\bigl\langle #1\,\vert\, #2\bigr\rangle}
\newcommand{\iinner}[2]{\left\llangle #1\,,\, #2\right\rrangle}
\newcommand{\Bigiinner}[2]{\Bigl\llangle #1\,,\, #2\Bigr\rrangle}
\newcommand{\Sinner}[3]{\left( #1\,\vert\, #2\right)^{S_{#3}}}
\newcommand{\bigSinner}[3]{\bigl( #1\,\vert\, #2\bigr)^{S_{#3}}}
\newcommand{\mi}{{\scriptscriptstyle{-1}}}	
\newcommand{\comm}[2]{\left[ #1, #2 \right]} 
\newcommand{\pv}[1]{\mathsf{pv} \left(\int #1\right)}

\begin{document}
\title{Elliptic Operators and $\Kth$\=/Homology}
\date{\today}
\address{Anna Duwenig\\
Department of Mathematics and Statistics\\
University of Victoria\\
Victoria, BC\\
Canada V8P 5C2}%
\author{Anna Duwenig}
\email{aduwenig@uvic.ca}

\begin{abstract}
	If a differential operator~$D$ on a smooth Hermitian vector bundle~$S$ over a compact manifold~$M$ is symmetric, it is essentially self\-/adjoint and so admits the use of functional calculus. If~$D$ is also elliptic, then the Hilbert space of square integrable sections of $S$ with the canonical left $C(M)$\-/action and the operator $\chi(D)$ for~$\chi$ a normalizing function is a Fredholm module, and its $\Kth$\=/homology class is independent of $\chi$. In this expository article, we provide a detailed proof of this fact following the outline in the book ``Analytic K-homology'' by Higson and Roe.
\end{abstract}

\thanks{The author would like to thank an anonymous referee for helpful remarks.}
%

\maketitle

\section{Introduction}

A differential operator $D$ acting on the sections of a smooth Hermitian vector bundle \mbox{$S\stackrel{\pi}{\to} M$} over a compact manifold $M$ can be regarded as an unbounded operator on the Hilbert space~$\mathrm{L}^2 (M; S)$ of square integrable sections of~$S$. If $D$ is symmetric, then it is automatically essentially self\-/adjoint and hence we can use functional calculus. If $D$ is also elliptic, then~$\mathrm{L}^2 (M; S)$ with the canonical left $C(M)$\-/action by multiplication and the operator $\chi(D)$ for $\chi$ a normalizing function turns out to be a Fredholm module over $C(M)$, whose $\Kth$\-/homology class $[D]$ is independent of the choice of $\chi$. The goal of this paper is to give the details of the proof of \cite[Thm.\ 10.6.5]{HigRoe:KHom} in the compact case in order to make it more accessible. In particular, we compile the definitions and constructions from \cite[\S~8, \S~9]{HigRoe:KHom} that are needed to understand the theorem, we elaborate on aspects which are sparse on details (notably the proofs of Propositions~10.3.1,
10.3.5, 10.6.2 in \cite{HigRoe:KHom}), and we provide complete solutions to two crucial steps, namely \cite[Exercise~10.9.1]{HigRoe:KHom} and \cite[Exercise~10.9.3]{HigRoe:KHom}.

We start in \Cref{sec:KHom} by defining the $\Kth$\-/homology groups of a $C^*$\-/algebra $A$. These groups consist of equivalence classes of triples $(\nu,\mathcal{H}, F)$ where $\nu$ is a representation of $A$ on the Hilbert space $\mathcal{H}$ and $F$ is a bounded operator on $\mathcal{H}$ with additional properties. If $A$ is unital, these can be stated as: $F$ is essentially self\-/adjoint, is essentially unitary, and essentially commutes with the left $A$\-/action.

In \Cref{sec:unbddops}, we construct the Cayley Transform for densely defined self\-/adjoint unbounded operators. We conclude that these operators allow the use of functional calculus.

In \Cref{sec:ellops}, we first survey differential operators and prove some of their properties; for example, what their commutator with a multiplication operator looks like and that we can define their symbol independently of the choice of charts. In \Cref{subsec:Sobolev}, we study Sobolov spaces in order to make sense of (but not prove) G{\aa}rding's Inequality. In \Cref{subsec:NormalizingAndFT}, we prove the existence of normalizing functions whose distributional Fourier Transforms are supported in arbitrarily small intervals around~$0$, as is stated in \cite[Exercise~10.9.3]{HigRoe:KHom}. This is needed to show that~$\chi(D)$ essentially commutes with the left action.

At this point, we are equipped to dive into the proof of the main theorem, cf.\ \Cref{thm:main}, which is the content of \Cref{sec:Thm-main}. 

The appendix contains a detailed proof of the existence of Friedrichs' mollifiers, cf.\ \cite[Exercise~10.9.1]{HigRoe:KHom}. This tool is important to show that $D$ is essentially self\-/adjoint, so that it makes sense to consider \mbox{$F=\chi(D)$} in the main theorem.

We should point out that the assumption that~$D$ be elliptic is needed solely to invoke G{\aa}rding's Inequality. Therefore, we will not dwell upon ellipticity of~$D$, despite it being crucial for the construction of the $\Kth$\-/homology class~$[D]$ and despite the title of this paper. 


\section{Kasparov's $\mathrm{K}$\-/homology}\label{sec:KHom}

\subsection{Gradings}\label{ssec:gradings}
The material of this \namecref{ssec:gradings} is from \cite[Appendix~A]{HigRoe:KHom}.

	A \textit{$\mathbb{Z}/ {2} \mathbb{Z}$\=/grading} of a vector space $V$ is a direct sum decomposition into two subspaces 
	$V=V^{+} \oplus V^{-}$, the vectorspace's \textit{even} and \textit{odd} part. We will often just say that $V$ is \textit{graded}. Equivalently, $V$ is equipped with a vector space automorphism $\gamma$ such that $\gamma^2 = \mathrm{id}_V$, and we obtain the decomposition as \mbox{$V^{\pm} = \{ v\in V\,|\, \gamma(v)=\pm v\}$.}
	An element $v\in V$ is called \textit{homogeneous} if it is in one of these two subspaces, and its \textit{degree} is defined by
	\[
	\partial v =
	\left\{
		\begin{array}{ll}
			0 & \mbox{if } v\in V^+, \\
			1 & \mbox{if } v\in V^-.
		\end{array}
	\right.
	\]
	
	We define $V^{\mathrm{op}}$ to be $V$ as vector space but with reversed grading, that is,
	\[
		(V^{\mathrm{op}})^\pm
		:=
		V^\mp.
	\]
	For an endomorphism $T$ of $V$, we write $T^{\mathrm{op}}$ when we consider it as an endomorphism of $V^{\mathrm{op}}$.	
	The direct sum of two graded spaces $V, W$ is equipped with the grading
	\[
		(V\oplus W)^+
		:=
		V^+ \oplus W^+
		\;\text{ and }\;
		(V\oplus W)^-
		:=
		V^- \oplus W^-
		.
	\]	
	A Hilbert space is \textit{graded} if it is graded as a vector space, and its even and odd subspaces are closed and mutually orthogonal. Equivalently, the grading automorphism $\gamma$ is a bounded unitary operator.
\begin{example}
	A grading of a Hilbert space~$\mathcal{H}$ induces a grading on $\B(\mathcal{H})$ by deeming an operator $T$ even (resp.\ odd) if $T$ preserves (resp.\ reverses) the two subspaces. In terms of the grading operator $\gamma$, $T$ is even (resp.\ odd) if and only if $T\circ\gamma = \gamma\circ T$ (resp.\  $T\circ\gamma = - \gamma\circ T$). If we think of $\B(\mathcal{H})$ as
	\[
	\left(
	\begin{matrix}
	\B(\mathcal{H}^+) & \B(\mathcal{H}^-, \mathcal{H}^+)\\
	\B(\mathcal{H}^+, \mathcal{H}^-) & \B(\mathcal{H}^-)
	\end{matrix}
	\right),
	\]
	we see that
	\[
	T \text{ even} \iff
	T =
	\left(
	\begin{smallmatrix}
	* & 0\\
	0 & *
	\end{smallmatrix}
	\right),
	\;\;\text{ and }\;\;
	T \text{ odd} \iff
	T =
	\left(
	\begin{smallmatrix}
	0 & *\\
	* & 0
	\end{smallmatrix}
	\right).
	\]

	Moreover,
	\begin{align}
		\B(\mathcal{H})^\pm \circ \B(\mathcal{H})^\pm \subseteq \B(\mathcal{H})^+
		\;\;\text{ and }\;\;
		\B(\mathcal{H})^\pm \circ \B(\mathcal{H})^\mp \subseteq \B(\mathcal{H})^-,
	\end{align}
	and the adjoint preserves the grading: $T$ is even (resp.\ odd) if and only if $T^*$ is even (resp.\ odd). This makes $\B(\mathcal{H})$ \textit{graded as a $C^*$\=/algebra}.
\end{example}

\begin{example}[{\cite[Def.~11.2.2]{HigRoe:KHom}}]\label{ex:CliffordAlg}	
	Another example of a graded $C^*$\=/algebra is the complex \textit{Clifford algebra}: the complex unital $*$\=/algebra $\mathbb{C}_n$ is generated by~$n$ elements~$\varepsilon_{1},\ldots, \varepsilon_{n}$ which satisfy
	\begin{align}\label[property]{CliffordMult}
		\varepsilon_{i} \varepsilon_{j} + \varepsilon_{j} \varepsilon_{i} = 0 \text{ for } i\neq j,
		\quad
		\varepsilon_{i}^* = -\varepsilon_{i},
		\quad
		\;\text{ and }\;
		\quad
		\varepsilon_{i}^2 = -1.
	\end{align}
	By deeming the basis $\{\varepsilon_{j_{1}}\cdots \varepsilon_{j_{k}}:j_{1} < \ldots <j_{k},0\leq k \leq n\}$ orthonormal, $\mathbb{C}_n$ becomes a Hilbert space. The left\-/action by multiplication is then a faithful $*$\=/representation of $\mathbb{C}_n$ on $\mathbb{C}_n$, which makes it a $C^*$\=/algebra. An element $\varepsilon_{j_{1}}\cdots \varepsilon_{j_{k}}$ is regarded as even (resp.\ odd) if $k$ is even (resp.\ odd).
\end{example}

\subsection{Fredholm modules}\label{ssec:Fmod}
For a separable $C^*$\=/algebra $A$, we recall the following definitions from \cite[\S 8.1, \S 8.2]{HigRoe:KHom}.
\begin{defn}
	A \emph{Fredholm module over $A$} is a triple $(\nu,\mathcal{H},F)$ consisting of
	\begin{enumerate}
		\item a representation $\nu\colon  A\to\B(\mathcal{H})$ on a separable Hilbert space $\mathcal{H}$, and
		\item an operator $F\in \B(\mathcal{H})$ such that
	\end{enumerate}
	\begin{equation}
		\quad\nu(a) ( F^* - F),\; \nu(a) ( F^2 - 1),\; [\nu(a), F]
		\tag*{are compact for all $a\in A$.}
	\end{equation}
\end{defn}

It is sometimes helpful to be more precise and call such triples \emph{ungraded} or \emph{odd} Fredholm modules, in order to distinguish them from \emph{graded} (sometimes also called \emph{even}) \emph{Fredholm modules}: 
	
\begin{defn}	
	A \textit{graded Fredholm module} is a Fredholm module $(\nu,\mathcal{H},F)$ over $A$ such that
	\begin{enumerate}\setcounter{enumi}{2}
		\item $\mathcal{H}$ is $\mathbb{Z}/ {2} \mathbb{Z}$\=/graded,
		\item the operator $F$ is odd, and all operators $\nu(a)$ are even.
	\end{enumerate}
\end{defn}

\begin{defn}\label{def:FredholmModStuff}
	\begin{enumerate}[label=\textup{(\arabic*)}]
		\item Two (graded) Fredholm modules are called \textit{unitarily equivalent} if there exists a (grading preserving) unitary isomorphism $U$ between the Hilbert spaces which intertwines the representations of $A$ and the distinguished bounded operators.
		\item 	An \textit{operator homotopy} between (graded) Fredholm modules $(\nu, \mathcal{H}, F_{0})$ and \linebreak[4]$(\nu, \mathcal{H}, F_{1})$ is a family  \mbox{$\{(\nu, \mathcal{H}, F_{t})\}_{t\in [0,1]}$} of (graded) Fredholm modules such that \mbox{$[0,1]\to \B(\mathcal{H}), \,t\mapsto F_{t},$} is norm continuous.
		\item
			We say that a (graded) Fredholm module $(\nu, \mathcal{H}, F')$ is a \textit{compact perturbation} of a (graded) Fredholm module $(\nu, \mathcal{H}, F)$ if the operator $\nu(a)(F-F')$ is compact for all $a\in A$.
		\item A (graded) Fredholm module $(\nu, \mathcal{H}, F)$ is called \textit{degenerate} if the operators
		\begin{equation*}
				\nu(a) ( F^* - F),\quad \nu(a) ( F^2 - 1),\quad [\nu(a), F]
		\end{equation*}
		are zero, and not just compact, for all $a\in A$.
	\end{enumerate}
\end{defn}

\begin{prop}\label{CompactPertImpliesHomotopic}
	Two \textup(graded\textup) Fredholm modules which are compact perturbations of one another are operator homotopic.
\end{prop}

\begin{proof}
	The straight line from the operator $F$ to its compact perturbation $F'$, given by \mbox{$F_{t} := (1-t) F + t F' $} for $t\in[0,1]$, can be quickly checked to give the claimed homotopy of (graded) Fredholm modules.
\end{proof}

\subsection{The K-Homology groups}\label{ssec:KHomgps}

Since the sum of two (graded) Fredholm modules --given by the direct sum of Hilbert spaces, of representation, and of operators-- is again a  (graded) Fredholm module, we arrive at the following \namecref{def:KHom} for the $\Kth$\=/homology groups:

\begin{defn}[{\cite[Def.\ 8.2.5]{HigRoe:KHom}}]\label{def:KHom}
	For a separable $C^*$\=/algebra $A$, let $\Kth^{0}(A)$ be the abelian group with one generator $[x]$ for each unitary equivalence class of \emph{graded} Fredholm modules over $A$, subject to the following relations:
	\begin{enumerate}[label=\textup{(\arabic*)}]
		\item If $x, y$ are two such Fredholm modules, then $[x\oplus y] = [x]+[y]$, and
		\item two operator homotopic modules give the same class.
	\end{enumerate}
	Similarly, let $\Kth^{1}(A)$ be the abelian group with one generator for each unitary equivalence class of \emph{ungraded} Fredholm modules over $A$, subject to the same relations.
\end{defn}

\begin{rmk}[{\cite[Prop.~8.2.8 and~8.2.10; Cor.~8.2.11]{HigRoe:KHom}}]
	The sum of the Fredholm modules $x=\bigl(\nu,\mathcal{H},F\bigr)$ and \mbox{$y=\bigl(\nu^\mathrm{op},\mathcal{H}^\mathrm{op},-F^\mathrm{op}\bigr)$} is homotopic to a degenerate. Since the degenerate modules are zero in $\Kth^{j}(A)$ for $j=0,1$, we conclude that the classes $[x]$ and \mbox{$[y]$} in $\Kth$-homology are each other's additive inverse.
%
	Consequently, every element of $\Kth^{j}(A)$ can be represented by a single (graded) Fredholm module.
\end{rmk}

\begin{rmk}[{\cite[Lemma~8.3.8]{HigRoe:KHom}}]
	For the definition of the $\Kth$\-/homology groups, one can restrict to those Fredholm modules $(\nu,\mathcal{H},F)$ for which $\nu (A)\mathcal{H}$ is dense in~$\mathcal{H}$. In view of \Cref{def:FredholmModStuff}, we will refrain from calling such Fredholm modules ``non\-/degenerate''. But in the case where $A$ is unital, we can call them \textit{unital} Fredholm modules, since $\overline{\nu (A)\mathcal{H}}=\mathcal{H}$ is then equivalent to $\nu(1_A)=\mathrm{id}_{\mathcal{H}}$.
\end{rmk}
	
\begin{rmk}[{\cite[Def.\ A.3.1, 8.1.11, 8.2.5; Prop.\ 8.2.13,~8.8.5]{HigRoe:KHom}}]
	It is possible to define `higher' $\Kth$-homology groups $\Kth^{-p}(A)$ of $A$ for $p>0$, built from Fredholm modules with the additional datum of a $p$-multigrading on the Hilbert space.
	This collection of groups satisfy Bott periodicity, that is, there exists an isomorphism \mbox{$\Kth^{-p}(A)\to \Kth^{-p-2}(A)$.} For our purposes, it will be sufficient to focus on  
	$\Kth^{0}$ and~$\Kth^{1}$.
\end{rmk}

For $B$ another separable $C^*$\=/algebra and a unital $*$\=/homomorphism $\alpha\colon  B \to A$, we can turn a (graded) Fredholm module $(\nu, \mathcal{H}, F)$ over $A$ into one over $B$ by considering $(\nu\circ \alpha, \mathcal{H}, F)$. This process respects addition and unitary equivalence, and hence descends to a map on the level of $K$\=/homology,
\[
	 \Kth^{j}(\alpha) = \alpha^*\colon  \Kth^{j} (A) \to \Kth^{j} (B),\quad j=0,1.
\]
It is easily checked that the assignment $A\mapsto \Kth^{j} (A)$, $\alpha\mapsto \alpha^*$, is a contravariant functor from the category of separable $C^*$\=/algebras to the category of abelian groups.

\section{Unbounded operators}\label{sec:unbddops}
\subsection{Terminology}\label{ssec:unbddop-terms}
		An \textit{unbounded operator} $D$ on a Hilbert space $\mathcal{H}$ is a linear map from a subspace $\dom  D \subseteq \mathcal{H}$ into $\mathcal{H}$. If $\dom  D$ is dense, then let
		\begin{align*}
			\dom  D^* 
			:&=
			\bigl\{
				\eta \in \mathcal{H}\;|\; \dom  D \ni \xi \mapsto \inner{D \xi}{\eta } \text{ is bounded}
			\bigr\}\\
			&=
			\bigl\{
				\eta \in \mathcal{H}\;|\; \exists \chi\in\mathcal{H} :  \forall  \xi  \in \dom  D : \inner{D \xi }{ \eta} = \inner{\xi }{\chi}
			\bigr\}
			.
		\end{align*}
		For $\eta  \in \dom  D^* $, define $D^* \eta $ to be the unique vector such that $\inner{D \xi}{\eta } = \inner{\xi}{D^*\eta }$ for all~$\xi \in \dom  D $. The operator~$D^*$ is linear on its domain, and is called the \textit{adjoint} of~$D$. 
		
		The operator $D$ is called ...
		\begin{itemize}
			\item \textit{closed} if the graph of $D$ is a closed subset of $\mathcal{H}\oplus \mathcal{H}$.
			\item \textit{closable} if the closure of its graph is the graph of a function. This function is then \textit{the closure $\overline{D}$ of $D$}.
			\item an \textit{extension} of an unbounded operator $D'$ if $\dom D' \subseteq \dom  D$ and $D = D'$ on $\dom  D' $.
			\item \textit{symmetric} if $\inner{D \xi}{\eta }= \inner{ \xi }{D \eta  }$ for all $ \xi, \eta \in \dom  D $; in other words, if $D^*$ extends $D$.
			\item \textit{self-adjoint} if $\dom  D^*  = \dom  D$ and $D^* \xi = D \xi$ for all $\xi \in \dom  D$.
			\item \textit{essentially self-adjoint} if $D$ is symmetric and 
			$\dom  \overline{D}\ =\dom  D^*$.
		\end{itemize}
	Note that every symmetric operator $D$ is closable, and satisfies $\inner{D \xi}{\xi } \in\RR$ for $\xi\in\dom D$.	Moreover, for such $D$, $\dom \overline{D}$ is sometimes called the \textit{minimal domain} of $D$ and $\dom {D^*}$ the \textit{maximal domain} of $D$.

	\begin{example}[{\cite[Example 3 in Chapter VIII, Section 1]{ReedSimon:Methods}}]\label{ex:unbb-op-on-T}
		On $\mathcal{H}=L^2 (\mathbb{R})$, the assignment $Df = -\mathbf{i}\frac{\partial f}{\partial r}$ with domain $C_0^\infty (\mathbb{R})$ is an unbounded operator which is symmetric and hence closable (see \cite[Chapter VIII, Section 2]{ReedSimon:Methods}). A slight variant of this example is the unbounded symmetric operator $Df = -\mathbf{i}\frac{\partial f}{\partial\theta}$ on $L^2 (\mathbb{T})$ with domain $C^\infty (\mathbb{T})$.
		
		Many other examples of unbounded operators on Hilbert spaces can be found in \cite[Chapter VIII]{ReedSimon:Methods}, including some pathological ones like the last example in Section~3 
		and Problem~4: 
		both discuss symmetric operators, one with uncountably many and one with no self-adjoint extensions.
	\end{example}

\subsection{The Cayley Transform and Borel functional calculus}\label{ssec:Cayley-FuncCalc}

\begin{lem}[{\cite[Thm.\ VIII.3]{ReedSimon:Methods}}; {\cite[I.7.3.3]{Blackadar:OA}}]\label{SR:ResSurjective}
	If $D$ is a symmetric and densely defined unbounded operator on $\mathcal{H}$, then $D$ is self\-/adjoint if and only if $D\pm\mathbf{i}$ are both surjective. Moreover, in that case, $(D\pm\mathbf{i})^\mi$ is everywhere defined and a bounded operator.
\end{lem}

\begin{proof}
	Regarding the equivalence, we will actually only be interested in the forward implication, so let us disregard the proof of the other direction.
	We will follow the explanation given in \cite[I.7.3.3]{Blackadar:OA}.
	
	As~$D$ is self\-/adjoint, it is closed and the domains of $(D\pm\mathbf{i})^*$ and $D\mp \mathbf{i}$ both coincide with $\dom D$. Since for all~$\xi ,\eta\in  \dom D$,
	\[
		\inner{(D\pm\mathbf{i}) \xi }{\eta}
		=
		\inner{D\xi }{\eta} \pm\inner{\mathbf{i} \xi }{\eta}
		=
		\inner{\xi }{D\eta} \mp \inner{\xi }{\mathbf{i} \eta}
		=
		\inner{\xi }{(D\mp\mathbf{i}) \eta},
	\]
	we see that~$D\mp\mathbf{i}$ satisfies the universal property that determines~$(D\pm\mathbf{i})^*$ uniquely, so $D\mp\mathbf{i} = (D\pm\mathbf{i})^*$.	
	\begin{claim}\label{Dpmi-injective}
		For~$\xi\in\dom D$, we have~$\norm{(D\pm\mathbf{i})\xi}\geq \norm{\xi}$, so that $D\pm\mathbf{i}$ is bounded below by~$1$. In particular,~$D\pm\mathbf{i}$ is injective.
	\end{claim}
	\begin{Claimproof}
			For~$\xi\in \dom (D\pm\mathbf{i})=\dom D $, we have because of~$D=D^*$
				\begin{align}\label{normDpmi=normDmpi}
					\norm{(D\pm\mathbf{i})\xi}^2
					&=
					\norm{D\xi}^2
					\pm
					\inner{\mathbf{i} \xi}{D\xi}
					\pm
					\inner{D\xi}{\mathbf{i} \xi}
					+
					\norm{\mathbf{i} \xi}^2
					=
					\norm{D\xi}^2
					+
					\norm{\xi}^2.
				\end{align}
			Injectivity is now clear.
	\end{Claimproof}
%

\begin{claim}\label{rangeclosed}
	Since~$D\pm\mathbf{i}$ is bounded below and~$D$ is closed, the range of~$D\pm\mathbf{i}$ is closed.
\end{claim}
\begin{Claimproof}
	A straightforward computation shows that~$D\pm\mathbf{i}$ is closed because~$D$ is. If~$T\xi_{n} \to \eta$ for~$T:=D\pm\mathbf{i}$ and some~$\xi_{n}\in\dom T = \dom D$, then~$(T\xi_{n})_{n}$ is a Cauchy sequence. The previous claim shows
	\[
		\norm{T(\xi_{n} - \xi_{m})} \geq \norm{\xi_{n} - \xi_{m}},
	\]
	so we see that~$(\xi_{n})_{n}$ is also Cauchy and hence converges to some~$\xi$. As~$T$ is closed and~$(\xi_{n}, T\xi_{n})_{n}$ is a sequence in its graph that converges, we must have~$\xi\in\dom T$ and \mbox{$T\xi_{n} \to T\xi$.}
\end{Claimproof}
%
%
\begin{claim}
	$D\pm\mathbf{i}$ has dense range.
\end{claim}
\begin{Claimproof}
	If~$\xi\in \mathrm{range}(D\pm\mathbf{i})^\perp$, then~$\inner{(D\pm\mathbf{i})\nu}{\xi}=0=\inner{\nu}{0}$ for all \mbox{$\nu\in\dom D$.} In particular,~$\xi$ is in~$\dom (D\pm\mathbf{i})^*$ with~$0=(D\pm\mathbf{i})^* \xi =(D\mp\mathbf{i})\xi $. Thus,
	\[
		\overline{\mathrm{range}(D\pm\mathbf{i})}
		=
		\mathrm{range}(D\pm\mathbf{i})^{\perp\perp}
		\supset
		\mathrm{ker}(D\mp\mathbf{i})^\perp.
	\]
	Since~$\mathrm{ker}(D\mp\mathbf{i})=\{0\}$ by \Cref{Dpmi-injective}, $D\pm\mathbf{i}$ thus indeed has dense range.
\end{Claimproof}

All in all, we have shown that~$D\pm\mathbf{i}$ is both injective on~$\dom D$ and surjective. Therefore, there exists a linear map \mbox{$(D\pm\mathbf{i})^\mi \colon \mathcal{H}\to\dom D\subseteq\mathcal{H}$} which is inverse to~$D\pm\mathbf{i}$. Lastly, since $\norm{(D\pm\mathbf{i})\xi}\geq \norm{\xi}$, we conclude~$\norm{(D\pm\mathbf{i})^\mi} \leq 1$.
\end{proof}
\setcounter{claim}{0}

\begin{defn}[{\cite[I.7.2.5. Def.]{Blackadar:OA}}]
	For $D$ a densely defined unbounded operator on $\mathcal{H}$, the \textit{spectrum} $\sigma(D)$ of $D$ is defined as
	\begin{align*}
		\sigma(D) := \mathbb{C}\setminus \bigl\{ z\in\mathbb{C}\,\vert\,	&D-z \text{ is injective on } \dom (D -z)= \dom D
		\\ &\text{with dense range, and } (D-z)^\mi \text{ is bounded}
		\bigr\}.
	\end{align*}
\end{defn}

\begin{rmk}\label{fullrange}
	Note that the proof of \Cref{SR:ResSurjective} also works for any other $z\in\mathbb{C}\setminus\RR$ in place of $\mathbf{i}$. Thus we have shown that, if $D$ is self\-/adjoint, $\sigma(D)\subseteq \RR$.	
	Also, it follows from \Cref{rangeclosed} that, if~$D$ is closed and $D-z$ is bounded below, then $D-z$ has closed range. So if $z\notin \sigma(D)$, then the range of $D-z$ is all of $\mathcal{H}$.	
\end{rmk}

\begin{defn}[cf.\ {\cite[I.7.4.1.\ ff.]{Blackadar:OA}}]\label{def:Cayley}
	We define
	\[
		c\colon \RR\to \SS^1\setminus\{1\},
		\quad
		c(t)=\frac{t+\mathbf{i}}{t-\mathbf{i}},
		\qquad
		\text{ with inverse }
		\quad
		c^\mi(z)=\mathbf{i} \frac{z+\mathbf{i}}{z-\mathbf{i}}.
	\]
	If $D$ is a densely defined self\-/adjoint operator on $\mathcal{H}$, then \Cref{SR:ResSurjective} shows that it makes sense to define
	\[
		c(D) := (D+\mathbf{i}) (D-\mathbf{i})^\mi\colon 
		\mathcal{H}\to\mathcal{H},
	\]
	and that this map is an isomorphism of $\mathcal{H}$. It is called the \textit{Cayley Transform of $D$}. From \Cref{normDpmi=normDmpi}, we see that
	\mbox{$
		\norm{ (D+\mathbf{i})\xi} = \norm{ (D-\mathbf{i})\xi}
	$,} 
	so $c(D)$ is even a unitary. Moreover, it does not have $1$ in its spectrum: if $c(D)\xi = \xi$, then for $\xi'= (D-\mathbf{i})^\mi \xi$ we have~\mbox{$(D+\mathbf{i})\xi' = (D-\mathbf{i})\xi'$,} that is~\mbox{$\mathbf{i} \xi' = -\mathbf{i} \xi'$.} Thus, $\xi'=0$ and hence $\xi = 0$, so that we have shown that $c(D)-1$ is injective. On the other hand, if $\eta\in\mathcal{H}$ is arbitrary, let $\xi := -\frac{1}{2}\mathbf{i} (D-\mathbf{i})\eta$ and compute
	\[
	\bigl(c(D) - 1\bigr) \xi
	=
	(D+\mathbf{i})(D-\mathbf{i})^\mi \xi - \xi
	= 
	-\frac{1}{2}\mathbf{i} (D+\mathbf{i})\eta + \frac{1}{2}\mathbf{i} (D-\mathbf{i})\eta
	=
	\eta,
	\]
	so we have shown that $c(D)-1$ is also surjective.
%
%

	Conversely, if $U$ is a unitary which does not have $1$ as eigenvalue, then $U-1$ has dense range: if $\xi\in \mathrm{range}(U-1)^\perp$, then $\inner{(U-1)\eta}{\xi}=0$ for all $\eta\in\mathcal{H}$, so $(U^* -1 ) \xi = 0$. Injectivity of $U-1$ then implies $\xi =0$.
	Therefore, the so\-/called \textit{inverse Cayley Transform} of $U$ defined by
	\[
		c^\mi(U) := \mathbf{i}(U+1)(U-1)^\mi
		\colon 
		\mathrm{range}(U-1) \to \mathcal{H},
	\]
	is densely defined.
	
\end{defn}
	
\begin{lem}[{\cite[3.5.\ Corollary]{Conway:FA}}]
	The inverse Cayley Transform of a unitary which does not have $1$ as eigenvalue is a self\-/adjoint operator.
\end{lem}

	\begin{proof}
		A quick computation shows that $c^\mi(U)$ is symmetric, so we only need to check that the domain of its adjoint is contained in \mbox{$\mathrm{range}(U-1)$.} If \mbox{$\xi\in\dom (c^\mi(U))^*$,} then there exists $\eta\in \mathcal{H}$ such that for all \mbox{$\nu'\in \mathrm{range}(U-1)$,} we have
		\[
			\inner{c^\mi(U) \nu'}{\xi} = \inner{\nu'}{\eta}.
		\]
		In other words, for every $\nu' = (U-1)\nu$,
		\[
			\inner{\mathbf{i}(U+1)\nu}{\xi} = \inner{(U-1)\nu}{\eta}.
		\]
		Since this holds for every $\nu\in\mathcal{H}$, it follows that $-\mathbf{i} (U^*+1)\xi = (U^* -1)\eta$. By applying $\mathbf{i} U$ to both sides, we get
		\[
			\xi + U \xi = (1+U) \xi = (1-U) \mathbf{i} \eta  = \mathbf{i} \eta - U\mathbf{i} \eta.		
		\]
		Rearranging and adding $\xi$ to both sides yields
		\[
				2\xi  = (\mathbf{i} \eta - U\mathbf{i} \eta -  U \xi) + \xi = (1-U) (\mathbf{i} \eta + \xi),
		\]
		so $\xi\in \mathrm{range}(U-1)$ as claimed.
	\end{proof}
If $1\notin \sigma(U)$, then it follows from our comment in \Cref{fullrange} that $c^\mi(U)$ is actually everywhere defined and bounded. One can check that
	\[
		c^\mi (c(D)) = D
		\;\text{ and }\;
		c ( c^\mi (U) ) =U,
	\]
so we have found:

\begin{prop}[{\cite[3.5.\ Corollary; 3.1.\ Theorem]{Conway:FA}}]
	The Cayley Transform is a bijective map from the densely defined, self\-/adjoint operators to the unitary operators which do not have $1$ as eigenvalue.
\end{prop}

The Cayley Transform makes it possible to extend the Borel functional calculus for normal operators to densely defined, self\-/adjoint operators. It has the following properties:

\begin{prop}[Functional Calculus; {\cite[I.7.4.5. Thm, I.7.4.7.~Def.]{Blackadar:OA}}]\label{FuncCalc}
For $D$ a densely defined, self\-/adjoint operator on $\mathcal{H}$, there exists a linear map
\begin{align*}
	\bigl\{h\colon \RR\to\mathbb{C} \text{ Borel measurable} \bigr\} &\longrightarrow \bigl\{\text{densely defined unbounded operators on $\mathcal{H}$} \bigr\}
		\\
	h&\longmapsto h(D)
\end{align*}
with the following properties:

\begin{enumerate}[label=\textup{(\arabic*)}]
	\item $\mathrm{id}_{\RR} (D) = D$.
	\item If $h\geq 0$, then $h(D)$ is positive.
	\item If $\abs{h}=1$, then $h(D)$ is unitary.
	\item\label[property]{FuncCalc-selfadj} $h(D)^* = \overline{h}(D)$; in particular, if $h$ is real\-/valued, then $h(D)$ is self\-/adjoint.
	\item\label[property]{FctCalc-norm} If $h$ is bounded and continuous, then $$\norm{h(D)}=\supnorm{h}.$$
	\item\label[property]{SR:FctCalcConvergence} If $h_{n}$ is a uniformly bounded sequence of functions which converges pointwise to $h$, then \mbox{$h_{n}(D)\to h(D)$} strongly.
\end{enumerate}
\end{prop}

\begin{lem}[special case of {\cite[Lemma~10.6.2]{HigRoe:KHom}}]\label{lem:10.6.2}
	Suppose $D$ is an unbounded, essentially self\-/adjoint operator on~$\mathcal{H}$, and~$T\in \B(\mathcal{H})$ preserves~$\dom D$ and satisfies~\mbox{$TD = -DT$.} If~$f\in C_{b} (\RR)$ is odd, then~\mbox{$Tf(D) = -f(D)T$,} and if~$f$ is even, then~\mbox{$Tf(D) = f(D)T$.}
\end{lem}

\noindent\textit{Proof.}
	Let us first set some notation:	the decomposition of a function $f$ into its even and odd part is given by
	\begin{align*}
		f^{\mathrm{e}} (x) = \frac{f (x) +f(-x)}{2}
		\quad\;\text{ and }\;\quad
		f^{\mathrm{o}} (x) = \frac{f (x) - f (-x)}{2},
		\qquad
		\text{ so that }
		f = f^\mathrm{e} + f^\mathrm{o}.
	\end{align*}
	Let us denote by
	\[
		\tilde{f} (x) := f^\mathrm{e} - f^\mathrm{o} = f(-x).
	\]	
	The claim can now be rephrased to $T f(D) = \tilde{f}(D) T$. In other words, $T$ graded commutes with $f(D)$ when $C_{b} (\RR)$ has the $\mathbb{Z}/ {2} \mathbb{Z}$\=/grading into even and odd functions.
	
	\begin{claim}
		It suffices to show the claim for elements of $C_{0} (\RR)$.
	\end{claim}
\begin{Claimproof}
			For $f\in C_{b} (\RR)$, take functions $f_{n}\in C_{0} (\RR)$ converging pointwise to $f$. By \Cref{SR:FctCalcConvergence} of Functional Calculus, we have strong convergence $f_{n}(D)\to f(D)$ and also $\tilde{f_{n}}(D)\to \tilde{f}(D)$, so for every $h\in\mathcal{H}$, we get
			\begin{align}
				T f(D) h
				=
				T \bigl(\underset{n\rightarrow\infty}{\mathrm{lim}}{f_{n}(D) h}\bigr)
				=
				\underset{n\rightarrow\infty}{\mathrm{lim}} {T f_{n}(D) h}
				=
				\underset{n\rightarrow\infty}{\mathrm{lim}} {\tilde{f_{n}}(D) T  h}
				=
				{\tilde{f}(D) T  h},
				\notag
			\end{align}
	where we used the assumption that $T f_{n}(D) = \tilde{f}_{n}(D) T$
\end{Claimproof}
	By the Stone\-/Weierstrass Theorem \cite{deBra:SWThm}, either of the functions
	\[
		\psi_\pm (x) := \frac{1}{\mathbf{i} \pm x} = \overline{\psi_\mp (x)}
	\]
	generate $C_{0} (\RR)$ as a $C^*$\=/algebra.
	
	\begin{claim}
		It suffices to show that $T$ graded commutes with $\psi_\pm(D)$.
	\end{claim}
	\begin{Claimproof}
		For a fixed $f\in C_{0} (\RR)$, assume   
		\[
			\psi = \sum_{n,k\in\NN^{\times}} a_{n,k} \psi_+^n	(\overline{\psi_+})^k=	\sum_{n,k\in\NN^{\times}} a_{n,k} \psi_+^n	{\psi}_-^k
			\quad\text{is such that}\quad
			\supnorm{f-\psi}<\epsilon .
		\]
		The properties of continuous functional calculus shows that, if~$T$ graded commutes with~$g(D)$ for $g$ some continuous function, then it also graded commutes with~$g^n(D)$ for positive powers of~$g$. Thus, we have $T\psi_{\pm}^n(D)=(\tilde{\psi}_{\pm})^n(D) T$ by assumption, which implies
		\begin{align*}
		\norm{Tf(D) - \tilde{f}(D) T}
		&\leq
		\norm{Tf(D) - T\psi (D)}
		+
		\norm{ T\psi(D)- \tilde{f}(D) T}\\
		&=
		\norm{T(f-\psi)(D)}
		+
		\norm{\tilde{\psi}(D) T- \tilde{f}(D) T}\\
		&\leq
		\norm{T}\cdot
		\supnorm{f - \psi}
		+
		\supnorm{\tilde{\psi} - \tilde{f}}
		\cdot \norm{T}
		< 2\epsilon \norm{T}.
		\end{align*}
		Since this is possible for any $\epsilon$, this implies $Tf(D) = \tilde{f}(D) T$ as wanted.
	\end{Claimproof}

	As $(\mathbf{i} \pm D) T = T (\mathbf{i} \mp D)$ by assumption, we get 
	\begin{align}\label{TpsipsiT}
		T \psi_{\mp} (D)= \psi_{\pm} (D) T .
	\end{align}

	Since $\psi_\pm (-x) = \psi_\mp (x)$, we can see that
	\begin{align*}
		\psi_{+}^{\mathrm{e}} 
		=
		\psi_{-}^{\mathrm{e}} 
		\quad\;\text{ and }\;\quad
		\psi_{+}^{\mathrm{o}} 
		=-
		\psi_{-}^{\mathrm{o}} 
		.
	\end{align*}

	As a consequence,
	\begin{align*}
	&2(\psi_+ + \psi_-)
	=
	(\psi_{+}^{\mathrm{e}} + \psi_{+}^{\mathrm{o}}) + (\psi_{-}^{\mathrm{e}} + \psi_{-}^{\mathrm{o}})
	=
	2\psi_{+}^{\mathrm{e}},
	\quad
	\text{so }
	\psi_+ + \psi_-
	=
	\psi_{+}^{\mathrm{e}},
	\\
	\;\text{ and }\;\quad &2(\psi_+ - \psi_-)
	=
	(\psi_{+}^{\mathrm{e}} + \psi_{+}^{\mathrm{o}}) - (\psi_{-}^{\mathrm{e}} + \psi_{-}^{\mathrm{o}})
	=
	2\psi_{+}^{\mathrm{o}},
	\quad
	\text{so }
	\psi_+ - \psi_-
	=
	\psi_{+}^{\mathrm{o}}.
\end{align*}

	From \Cref{TpsipsiT}, it thus follows that
	\[
		T \psi_{+}^\mathsf{e} (D)
		=
		T \bigl(\psi_+ + \psi_-\bigr)(D) 
		=
		\bigl(\psi_- + \psi_+\bigr)(D) T
		=
		\psi_{+}^\mathsf{e} (D) T
	\]
	and
	\[
		T \psi_{+}^\mathsf{o} (D)
		=
		T \bigl(\psi_+ - \psi_-\bigr)(D) 
		=
		\bigl(\psi_- - \psi_+\bigr)(D) T
		=
		- \psi_{+}^\mathsf{o} (D) T.
	\]
	In other words, $T$ graded commutes with $\psi_+ (D)$.
\qed

\setcounter{claim}{0}

\section{Elliptic operators} \label{sec:ellops}
\begin{notation}
	We will write $\lambda$ for Lebesgue measure on $\RR^n$, $\norm{\,\cdot\,}_{\mathbb{C}^k}$ for the Euclidean norm on $\mathbb{C}^k$, and $\norm{\,\cdot\,}_{2}$ for $\mathrm{L}^2$\=/norms.
\end{notation}

\begin{defn}
	A vector bundle $S\stackrel{\pi}{\to} M$ over a smooth manifold~$M$ is called \textit{smooth} if~$S$ is also a manifold and~$\pi$ is a smooth map. We write~$\Gamma (M; S)$ for the sections of this bundle, i.e.\
	\[
		\Gamma (M; S)
		:=
		\left\{
			v\colon M \to S
			\;\vert\;
			v_p \in S_p \text{ for all } p\in M
		\right\},	
	\]
	and we write $\Gamma^\infty (M; S)$ resp.\ $\Gamma_{c} (M; S)$ for the smooth resp.\ compactly supported sections.
	
	A smooth vector bundle~$S\stackrel{\pi}{\to} M$ is called
	\textit{Hermitian} if, for each~$p\in M$, there is an inner product~$\Sinner{\cdot}{\cdot}{p}$ on the fibre~$S_{p} := \pi^\mi(p)$, and these inner products \textit{vary smoothly}: for every $u, v\in \Gamma^\infty (M; S)$, the map
	\[
	M\ni p \mapsto \bigSinner{u(p)}{v(p)}{p} \in \mathbb{C}
	\]
	is smooth. 
\end{defn}

	In the following, we will fix a smooth Hermitian complex vector bundle~$S\stackrel{\pi}{\to} M$ of rank~$k$ over a smooth manifold~$M$ of dimension~$n$. Let us denote the norm induced by the inner product~$\Sinner{\cdot}{\cdot}{p}$ on $S_{p}$ by~$\norm{\,\cdot\,}_{S_{p}}$. An example to keep in mind is the case where $M$ is spin$^c$ and $S$ is its spinor bundle.

	We further assume that we are given a \textit{nowhere\-/vanishing smooth measure} $\mu$ on $M$, that is, $\mu$ is a Borel measure	such that for every chart $(U, \varphi)$ of $M$, there exists a smooth function $f:\varphi(U)\to (0,\infty)$ such that \mbox{$\mathrm{d}(\varphi_{\ast} \mu_{U}) = f \,\mathrm{d}\lambda _{\varphi(U)}$.}
	This means for a \mbox{$(\varphi_{\ast} \mu_{U})$}\-/integrable function $h\colon  \varphi(U) \to \mathbb{C}$ that
	\begin{align*}
		\int\limits_{U} h\circ\varphi \,\mathrm{d}\mu  = \int\limits_{\varphi (U)} h \cdot f \,\mathrm{d}\lambda .
	\end{align*}
	Moreover, since $f$ does not vanish, we can also consider $g = \frac{1}{f}$ and get for \linebreak[4]\mbox{$\lambda_{\varphi(U)}$}\=/integrable $h$
	\begin{align}\label{RNderivative-g}
		\int\limits_U (h\cdot g) \circ \varphi \,\mathrm{d}\mu  = \int\limits_{\varphi (U)} h \,\mathrm{d}\lambda .
	\end{align}

\begin{rmk}\label{rmk:standAssu-L}
	For technical reason, there will be the standing assumption that there exists a number~$L$ so that we have for all of the above mentioned Radon\-/Nikodym derivatives the inequality~$\supnorm{f},\supnorm{g}\leq L$.
\end{rmk}
	
	We construct the Hilbert space $\mathrm{L}^2 (M; S)$ as the completion of $\Gamma_{c}^{\infty} (M; S)$ with respect to the norm coming from the inner product
	\[
		\inner{u}{v} := \int\limits_M \Sinner{u(p)}{v(p)}{p} \;\mathrm{d}\mu  (p).
	\]
	For a subset~$U\subseteq M$, we will write $\mathrm{L}^2 (U; S)$ for the completion of the smooth sections whose compact support is contained in~$U$. Lastly, let $$\M{}\colon  C_{0} (M) \to \B\bigl(\mathrm{L}^2 (M; S)\bigr),\quad g\mapsto \M{g},$$ be the representation of~$C_{0} (M)$ which, on the dense subspace~$\Gamma_{c}^{\infty} (M; S)$, is given by pointwise multiplication.
		

\subsection{Differential operators}\label{subsec:diffops}

\begin{defn}
	A \textit{\textup{(}first order linear\textup{)} differential operator acting on the sections of~$S$} is a~$\mathbb{C}$\-/linear map
	\[
		D \colon  \Gamma^\infty (M;S)\to \Gamma^\infty (M;S)
		\qquad
		\text{such that}
	\]
	\begin{description}[style=multiline, labelwidth=.5cm, leftmargin=0.5in]
		\item[\namedlabel{DiffOpPreservesSupp}{a)}] 
		 if $u, v\in \Gamma^\infty (M;S)$ agree on an open set $U$, then $Du, Dv$ also agree on $U$, and
		\item[b)] for a coordinate chart of $M$ that also trivializes $S$, say
					\begin{equation}\label[Diagram]{ChartTriv}
						\begin{tikzcd}
							\hphantom{M \supseteq}S_{| U }
							\arrow[rr, "\psi", bend left, dashed]
							\arrow[d, "\pi"', shift left=2.2ex]
							\arrow[dr, phantom, near start, "\circlearrowleft"]
							\arrow[r, "\approx"', "\Psi"]{} &  U \times \mathbb{C}^k
							\arrow[dl, "\mathrm{pr}_{1}"]{}
							\arrow[r, "\mathrm{pr}_{2}"'] & \mathbb{C}^k \\
							M \supseteq  U    	\arrow[r, "\approx"', "\varphi"]{}	&	V \subseteq \RR^n &
						\end{tikzcd}
					\end{equation}
					there exist functions $A^1, \ldots, A^n, B\in C^\infty \bigl(U, \mathrm{M}_{k} (\mathbb{C})\bigr)$ such that for all~$p\in U$ and all~$u\in\Gamma^\infty (M;S)$, we have
					\begin{align}
					\begin{split}\label{DiffOpLook}
						\bigl(Du\bigr) (p)
						=
						\sum_{j=1}^{n}
						&\Psi^\mi
						\left(
							p,
							A^j (p) \cdot {\partial_{j} (\psi \circ u\circ \varphi^\mi)}_{|\varphi (p)}
						\right)\\
						+
						&\Psi^\mi
						\bigl(
							p,
							B (p) \cdot (\psi \circ u)(p)
						\bigr).
					\end{split}
					\end{align}
	\end{description}
\end{defn}

We will from now on regard such a differential operator as an unbounded operator on $\mathrm{L}^2 (M;S)$ with dense domain $\Gamma_{c}^\infty (M;S)$. By abuse of terminology, we will say ``differential operator on $M$'', tacitly assuming a fixed Hermitian bundle $S$.

	\begin{example}[{\cite[Chapter~3]{Roe:Elliptic}}; {\cite[Chapter~II, Section \S 5]{LM:SpinGeom}}]\label{ex:Dirac-on-mfd-is-symmetric}
		Suppose one has a smooth Hermitian bundle $S$ over a manifold $M$, consisting of Clifford $TM$-modules and equipped with a connection, $\nabla$. Then one can locally define a \emph{Dirac operator} $\Dirac{M}$ by
		\begin{align}\label{eq:Dirac-locally}
			(
			\Dirac{M} u
			) (p)
			=
			\sum_{i=1}^{n}
			c_p
			\left(
				\tfrac{\partial\;}{\partial \varphi^{i}}_{\vert p}
			\right)
			\cdot
			\nabla_{\!\tfrac{\partial\;}{\partial \varphi^{i}}} (u)_p
			,
		\end{align}
		where $u$ is a smooth compactly supported section of $S$, $\varphi$ is a chart of $M$ around $p$, and $c$ denotes the Clifford action of the tangent vector $\tfrac{\partial\;}{\partial \varphi^{i}}$ on $S$. We immediately see that $\Dirac{M}$ is a first order differential operator. One can further show (see \cite[Prop. 3.11]{Roe:Elliptic}) that $\Dirac{M}$ is symmetric.
		
		A bundle $S$ with such structure would be the spinor bundle of a spin$^c$ manifold. In the example $M=\mathbb{T}$ with its canonical spin$^c$ structure, the spinor bundle is the trivial line bundle, $\mathbb{T}\times\mathbb{C}$, so the domain, $\Gamma^\infty_c(M;S)$, of $\Dirac{\mathbb{T}}=-\mathbf{i}\frac{\partial\;}{\partial\theta}$ is then just $C^\infty(\mathbb{T})$, smooth functions on the circle.
	\end{example}
	
\begin{lem}\label{AdjDPresveresSupp}
	Let~$D$ be a symmetric differential operator on $M$ 
	and let $u\in \dom D^*$ have compact support $K$. Then the support of $D^* u$ is contained in $K$.
\end{lem}
\begin{proof}
	Let $w_{k}$ be a sequence in $\dom D = \Gamma_{c}^{\infty} (M; S)$ which converges to $u$ in $\mathrm{L}^2$\-/norm.  If we take $K = \bigcap_{k=1}^{\infty} V_{k}$ for open nested sets $V_{k+1}\subseteq V_{k} \subseteq M$ (see \Cref{lem:Cpct-as-nested-Gdelta} for a construction), Urysohn gives us smooth $[0,1]$\-/valued functions $\rho_{k}$ with $\mathsf{supp}(\rho_{k})\subseteq V_{k}$ which are $1$ on $K$. Note that $u_{k} := \rho_{k}\cdot w_{k}$ is also in $\dom D$, and since $u$ is supported in $K$, we see
	 \begin{align*}
	 	\norm{u - u_{k}}^2_{2}
	 	&= \int\limits_{K} \norm{u(p) - u_{k}(p)}^2_{S_{p}} \,\mathrm{d}\mu 
	 	+ \int\limits_{M\setminus K} \norm{u_{k} (p)}^2_{S_{p}} \,\mathrm{d}\mu \\
	 	&\leq
	 	\int\limits_{K} \norm{u(p) - w_{k}(p)}^2_{S_{p}} \,\mathrm{d}\mu 
	 	+ \int\limits_{M\setminus K} \norm{w_{k}(p)}^2_{S_{p}} \,\mathrm{d}\mu 
	 	=
	 	\norm{u - w_{k}}^2_{2},
	\end{align*}
	 so $u_{k}$ also converges to $u$.
	 As $u_{k}$ is supported in $V_{k}$, we get from Property~\hyperref[DiffOpPreservesSupp]{a)} of differential operators that $D u_{k}$ is supported in $V_{k}$, too. We know that $D u_{k} = D^* u_{k}$ converges to $D^* u$ in $\mathrm{L}^2$\-/norm, so by choosing an appropriate subsequence, we can assume $(*)$ in the following computation:	 
	 \begin{align*}
	 	\frac{1}{k} 
	 	&\stackrel{(*)}{>}
	 	\norm{D^*u - Du_{k}}_{2}^2=
	 	\int\limits_{V_{k}} \norm{D^*u(p) - Du_{k}(p)}^2_{S_{p}} \,\mathrm{d}\mu 
	 	+ \int\limits_{M\setminus V_{k}} \norm{D^*u(p)}^2_{S_{p}} \,\mathrm{d}\mu \\
	 	&\geq
		\int\limits_{M\setminus V_{k}} \norm{D^*u(p)}^2_{S_{p}} \,\mathrm{d}\mu .
	\end{align*}
	 Now, note that $V_{k+m}\subseteq V_{k}$ for any $m$, and hence
	 \[
	 	\int\limits_{M\setminus V_{k}} \norm{D^*u(p)}^2_{S_{p}} \,\mathrm{d}\mu  \leq \int\limits_{M\setminus V_{k+m}} \norm{D^*u(p)}^2_{S_{p}} \,\mathrm{d}\mu  < \frac{1}{k+m} .
	 \]
	 It follows that $\int\limits_{M\setminus V_{k}} \norm{D^*u(p)}^2_{S_{p}} \,\mathrm{d}\mu =0$ for every $k$, 
	 and as $M\setminus K = \bigcup_{k} M \setminus V_{k}$,
	 \[
	 	\int\limits_{M\setminus K} \norm{D^*u(p)}^2_{S_{p}} \,\mathrm{d}\mu 
	 	\leq 
	 	\sum_{k} \int\limits_{M\setminus V_{k}} \norm{D^*u(p)}^2_{S_{p}} \,\mathrm{d}\mu 
	 	=
	 	0.
	 \]
	 We conclude that $D^*u$ is also supported in $K$.
\end{proof}

\begin{lem}\label{commDMisbdd} 
	If $D$ is a differential operator on $M$  which is locally given by \Cref{DiffOpLook}, and if $g\in C^\infty (M)$, then $\comm{D}{\M{g}}$ can locally be written as
	\begin{align}\label{eq:commDM}
		\comm{D}{\M{g}}u (p)
		&=
		\sum_{j=1}^{n}
		\partial_{j} \bigl( g \circ\varphi^\mi\bigr)_{|\varphi (p)} \cdot 
		\Psi^\mi
		\left(
		p,
		A^j (p) 
		\cdot  (\psi  \circ u (p))
		\right)
		.
	\end{align}
	In particular, if $K\subseteq M$ is compact, then $\comm{D}{\M{g}}$ extends to a bounded operator on \mbox{$\mathrm{L}^2 (K;S)$.}
\end{lem}
\begin{proof}
	It suffices to consider those $D$ that locally look like only one of the summands in \Cref{DiffOpLook}. Given a chart $(U,\varphi)$ and a trivialization $\Psi$ of $S$, if 
	\[
		\bigl(Du\bigr) (p)
		=
		\Psi^\mi
		\bigl(
			p,
			B (p) \cdot (\psi \circ u)(p)
		\bigr),
		\quad
		B\in C^\infty \bigl(U, \mathrm{M}_{k} (\mathbb{C})\bigr),
	\]
	then $D$ is itself only a multiplication operator (albeit by a matrix), and so it in fact commutes with $\M{g}$. So consider the case in which
	\[
		\bigl(Du\bigr) (p)
		=
		\Psi^\mi
		\left(
			p,
			A (p) \cdot {\partial_{j} (\psi \circ u\circ \varphi^\mi)}_{|\varphi (p)}
		\right),\quad
		A\in C^\infty \bigl(U, \mathrm{M}_{k} (\mathbb{C})\bigr),
		\]
	for some $1\leq j \leq n$. We compute for $u\in \Gamma^\infty (M;S)$ and $p\in U$:
	\begin{align*}
		\comm{D}{\M{g}}u (p)
		&=
		D \bigl(  g  u\bigr) (p)
		- 
		g  (p) \bigl(Du\bigr)(p)\\
		&=
		\Psi^\mi
		\left(
		p,
		A (p) \cdot {\partial_{j} \bigl(\psi  \circ ( g  u)\circ \varphi^\mi\bigr)}_{|\varphi (p)}
		\right)\\
		&\qquad- g  (p)
		\Psi^\mi
		\left(
		p,
		A (p) \cdot {\partial_{j} (\psi  \circ u\circ \varphi^\mi)}_{|\varphi (p)}
		\right) 
		.
	\end{align*}
	As $ g  (p)$ is just a scalar and $\Psi^\mi\left(p,\,\cdot\,\right)$ and $\psi $ are linear, we get
	\begin{align*}
		\comm{D}{\M{g}}u (p)
		=\Psi^\mi
		&\left(
		p,
		A (p) \cdot {\partial_{j} \bigl((g\circ\varphi^\mi) \cdot (\psi  \circ  u\circ \varphi^\mi)\bigr)}_{|\varphi (p)}
		\right)
		\\
		- 
		\Psi^\mi
		&\left(
		p,
		g  (p)\cdot A (p) \cdot {\partial_{j} (\psi  \circ u\circ \varphi^\mi)}_{|\varphi (p)}
		\right).
	\end{align*}
	By the product rule,
	\[
		\partial_{j} \bigl((g\circ\varphi^\mi) \cdot (\psi  \circ  u\circ \varphi^\mi)\bigr)_{|\varphi (p)}
		\!
		=
		{\partial_{j} \bigl( g \circ\varphi^\mi\bigr)}_{|\varphi (p)} (\psi  \circ u (p))
		+
		 g (p) {\partial_{j} \bigl(\psi  \circ u \circ \varphi^\mi\bigr)}_{|\varphi (p)},
	\]
	so we arrive at 
	\begin{align}\label{commDMlocally}
		\comm{D}{\M{g}}u (p)
		&=
		\Psi^\mi
		\left(
		p,
		{\partial_{j} \bigl( g \circ\varphi^\mi\bigr)}_{|\varphi (p)} \cdot A (p) 
		\cdot  (\psi  \circ u (p))
		\right).
	\end{align}
\end{proof}

\begin{defn}
	The \textit{symbol $\sigma_{D}$} of a differential operator $D$ is the $\RR$\-/vector bundle morphism
	\[
		\sigma_{D}\colon 
		T^*M \to \End (S)
	\]
	defined as follows: given a cotangent vector $\xi\in T^*_{p} M$ at $p$, take a chart $(U,\varphi)$ around $p\in M$ and a trivialization of $S_U$ as in \Cref{ChartTriv}. Suppose $D$ locally looks as in \Cref{DiffOpLook}, and write \mbox{$\xi = \sum_{j=1}^n \xi_{j} \mathrm{d}\varphi^j_{p}$,} where \mbox{$\{\mathrm{d}\varphi^j_{p}\}_{j}$} denotes the basis of $T^*_{p} M$ that is dual to the basis \mbox{$\{\frac{\partial}{\partial\varphi^j}|_{p}\}_{j}$} of $T_{p} M$. Then we define for $\eta\in S_{p}$,
	\[
		\sigma_{D} (p, \xi)\eta
		:=
		\Psi^\mi
		\left(p,
			\sum_{j=1}^{n}
			\xi_{j} A^j (p) \psi(\eta)
		\right).
	\]
\end{defn}

\begin{rmk}
	In \Cref{commDMisbdd}, we have actually shown that 
	\begin{align*}
		\comm{D}{\M{g}}u (p)
		=
		\sigma_{D} (p, \mathrm{d} g_{|p}) \bigl( u(p)\bigr)
		.
	\end{align*}
\end{rmk}

\begin{lem}
	The definition of $\sigma_{D}$ does not depend on the choice of $\Psi$ or $\varphi$.
\end{lem}

\begin{proof}
	First, assume that $\Omega$ is another trivialization of $S_U$, and let $\omega:= \mathrm{pr}_{2} \circ \Omega$. Since the fibre maps of both $\Psi$ and $\Omega$ are linear isomorphisms, there exists a smooth map
	\[
	 	H\colon  U\to \mathrm{GL}_{k} (\mathbb{C})\qquad \text{given by}
		\begin{tikzcd}
			&
			\mathbb{C}^k
			\arrow[rr, bend right, dashed, "H(p)"']
			&
			\arrow[l, "{\Omega (p,\,\cdot\,)}"', "\cong"]
			S_{p}
			\arrow[r, "{\Psi(p,\,\cdot\,)}", "\cong"']
			&
			\mathbb{C}^k
		\end{tikzcd}.
	\]
	Moreover, we can write $D$ also in the form
	\begin{align*}
		\bigl(Du\bigr) (p)
		=
		\sum_{j=1}^{n}
		\,&\Omega^\mi
		\left(
		p,
		E^j (p) \cdot {\partial_{j} (\omega \circ u\circ \varphi^\mi)}_{|\varphi (p)}
		\right)\\
		+
		&\Omega^\mi
		\left(
		p,
		E (p) \cdot (\omega \circ u)(p)
		\right),
	\end{align*}
	for all $u\in\Gamma^\infty (M;S)$.	By clever choices of $u$ and some use of the product rule, one can conclude that
	\begin{equation*}
		A^j (p)
		=
		H(p) E^j (p) H^\mi (p)
	\end{equation*}
	for each $1\leq j \leq n$.	Therefore, for any $\eta\in S_{p}$,
	\[
		A^j (p) \psi(\eta)
		=
		H(p) E^j (p) \omega(\eta)
	\]
	and so
	\[
		\Psi^\mi \left(p, A^j (p) \psi(\eta)\right)
		=
		\Psi^\mi \left(p, H(p) E^j (p) \omega(\eta) \right)
		=
		\Omega^\mi \left(p, E^j (p) \omega(\eta) \right).
	\]
	We see from this that $\sigma_{D} (p, \xi)$ does not depend on the choice of $\Psi$.
	
	Next, let $\gamma$ be another chart around $p$. Again, we can write $D$ in the form
	\begin{align*}
	\bigl(Du\bigr) (p)
		=
		\sum_{l=1}^{n}
		\,&\Psi^\mi
		\left(
		p,
		F^l (p) \cdot {\partial_l (\psi \circ u\circ \gamma^\mi)}_{|\gamma (p)}
		\right)\\
		+
		&\Psi^\mi
		\left(
		p,
		F (p) \cdot (\psi \circ u)(p)
		\right).
	\end{align*}
	We get that
	\begin{align*}
		\sum_{j=1}^{n}
		A^j (p) \cdot {\partial_{j} (\psi \circ u\circ \varphi^\mi)}_{|\varphi (p)}	
		&=
		\sum_{l=1}^{n}
		F^l (p) \cdot {\partial_l (\psi \circ u\circ \gamma^\mi)}_{|\gamma (p)}
		\\
		&=
		\sum_{l=1}^{n}
		F^l (p) \cdot
		\sum_{j=1}^n
		{\partial_{j} (\psi \circ u\circ \varphi^\mi)}_{|\varphi (p)}
		{\partial_l (\varphi \circ \gamma^\mi)^j}_{|\gamma (p)},
	\end{align*}
	and so another clever choice of $u$ yields
	\begin{align*}
		A^j (p)
		=
		\sum_{l=1}^{n}
		{\partial_l (\varphi \circ \gamma^\mi)^j}_{|\gamma (p)}
		F^l (p).
	\end{align*}
	Moreover, if $\xi = \sum_l \nu_l \;\mathrm{d}\gamma^l_{p}$, then
	\[
		\nu_l = \xi \left( {\frac{\partial}{\partial\gamma^l}}_{|p}\right)
		=
		\sum_{j=1}^n \xi_{j} {\partial_l \left(\varphi\circ\gamma^\mi\right)^j}_{|\gamma (p)}.
	\]	
	Combined, we have for any $v\in \mathbb{C}^k$
	\[
		\sum_{j=1}^{n}
		\xi_{j} A^j (p)  v 
		=
		\sum_{j=1}^{n}
		\xi_{j}
		\left(
			\sum_{l=1}^{n}
			{\partial_l (\varphi \circ \gamma^\mi)^j}_{|\gamma (p)}
			F^l (p)
		\right)
		 v 
		=
		\sum_{l=1}^{n}
		\nu_l
			F^l (p)
		 v 	,
	\]
	and so we conclude that $\sigma_{D} (p, \xi)$ also does not depend on the choice of $\varphi$.
\end{proof}

\begin{defn}
	We say that a differential operator is \textit{elliptic} if its symbol $\sigma_{D}$ maps each \mbox{$(p, \xi)$} in $T^*M$ with $\xi\neq 0$ to an invertible endomorphism of $S_{p}$.
\end{defn}

\begin{example}\label{ex:Dirac-is-elliptic}
	The Dirac operator we mentioned in \Cref{ex:Dirac-on-mfd-is-symmetric} is elliptic: using \Cref{eq:Dirac-locally}, one can show that its symbol is given by
	\[
		\sigma_{\Dirac{M}} (p, \xi)^2
		=
		-\norm{\xi}^2,
	\]
	see \cite[11.1.1 Def.]{HigRoe:KHom} or \cite[Lemma 5.1]{LM:SpinGeom}.
\end{example}


\subsection{Sobolev Spaces}\label{subsec:Sobolev}

We want to construct the Sobolev space associated to our vector bundle. Recall first that for $f\in C_{c}^\infty (\RR^n,\mathbb{C})$, the Sobolev norm is defined by
\[
	\norm{f}^2_{1,\RR^n} :=
	\norm{f}_{2}^2 +
	\sum_{i=1}^{n} \norm{\frac{\partial f}{\partial x_{i}}}_{2}^2
	.
\]

Take an atlas of $M$ whose charts are small enough to also allow smooth, fibrewise isometric trivializations as in \Cref{ChartTriv}. For a compact subset $K$ of $M$, let $\{(U_{i}, \varphi_{i})\}_{i=1}^l$ be a subcover of charts, and denote the corresponding trivialisations of $S$ by $\Psi_{i} = \pi\times\psi_{i}$. These induce maps
	$$\Psi_{i}^*\colon  \Gamma^\infty (U_{i}; S_{| U_{i} }) \to \bigl(C^\infty(V_{i})\bigr)^k$$ which send a section $v\colon  U_{i} \to S_{| U_{i} }$ to the map
\[
	\RR^n \supseteq V_{i} \ni x \mapsto
	\psi_{i} \Bigl( v \bigl(\varphi_{i}^\mi (x)\bigr) \Bigr).
\]
As explained in \Cref{po1-for-compact}, we can pick smooth compactly supported functions 
\[
	\rho_{1},\ldots, \rho_l\colon  M \to [0,1]
	\text{ such that }
	\mathsf{supp}(\rho_{i})\subseteq U_{i}
	\;\text{ and }\;
	\sum_{i=1}^l \rho_{i}(p) = 1
	\text{ for } p\in K.
\]
We define for $u\in \Gamma^\infty(K;S)$ (that is, sections of the bundle supported in $K$):
\[
	\norm{u}_{1} := \sum_{i=1}^l \norm{\Psi_{i}^* ( \rho_{i} \cdot u )}_{1,\RR^n}.
\]
Though this norm relies heavily on the choices involved, its equivalence class does not.  We define $\mathrm{L}^2_{1} (K;S)$ to be the completion of $\Gamma^\infty(K;S)$ with respect to this norm. Let us gather some facts about Sobolev spaces that we will need later:

\begin{lem}\label{lem:sobolev-norm-bigger}
	For $K\subseteq M$ compact, there exists a number $c>0$ such that for all $u\in \mathrm{L}^2_{1} (K;S)$,
	\[
		\norm{u}_{2} \leq c \norm{u}_{1}.
	\]
\end{lem}

\begin{proof}
	Since $\norm{f}_{1,\RR^n} \geq \norm{f}_{2}$, we get
	\[
		\norm{u}_{1} 
		\geq
		\sum_{i=1}^l \norm{\Psi_{i}^* ( \rho_{i} \cdot u )}_{2}.
	\]
	
	Let $f_{i}, g_{i} = \frac{1}{f_{i}}$ be as in \Cref{RNderivative-g} for $(U_{i}, \varphi_{i})$. Recall that we assumed in \Cref{rmk:standAssu-L} that $\supnorm{f_{i}}\leq L$ for some number $L$ and all $i$. For $v\in \Gamma^\infty_{c} (U_{i}; S_{| U_{i} })$, we have
	\[
		 \norm{\Psi_{i}^* ( v )}_{2}^2
		 =
		 \int\limits_{\RR^n} \norm{\psi_{i} \Bigl( v \bigl(\varphi_{i}^\mi (x)\bigr) \Bigr)}_{\mathbb{C}^k}^2 \,\mathrm{d}\lambda 
		 =
		 \int\limits_{U_{i}} \norm{\psi_{i} \bigl( v (p) \bigr)}_{\mathbb{C}^k}^2 (g_{i}\circ\varphi)(p)\,\mathrm{d}\mu ,
	\]
	and since $\psi_{i}$ is isometric, we get 
	\[
		\norm{\Psi_{i}^* ( v )}_{2}^2
		\geq
		\frac{1}{L}
		\int\limits_{U_{i}} \norm{v (p)}_{S_{p}}^2\,\mathrm{d}\mu 
		=
		\frac{1}{L} \norm{v}_{2}^2
		.
	\]
	Thus,
	\[
		\norm{u}_{1} 
		\geq
		\frac{1}{\sqrt{L}} \sum_{i=1}^l \norm{\rho_{i} \cdot u }_{2}.
	\]
	Furthermore,
	\begin{align*}
		\left(\sum_{i=1}^l \norm{\rho_{i} \cdot u }_{2}\right)^2
		&\geq
		\sum_{i=1}^l \norm{\rho_{i} \cdot u }^2_{2}
		=
		\sum_{i=1}^l \left( \int\limits_{K} \rho_{i}(p)^2 \norm{ u  (p)}_{S_{p}}^2\,\mathrm{d}\mu \right)\\
		&=
		\int\limits_{K} \left(\sum_{i=1}^l \rho_{i}(p)^2\right) \norm{u  (p)}_{S_{p}}^2\,\mathrm{d}\mu \\
		&\geq
		\int\limits_{K} \frac{1}{l}\left(\sum_{i=1}^l \rho_{i}(p)\right)^2 \norm{u  (p)}_{S_{p}}^2\,\mathrm{d}\mu 
		=
		\frac{1}{l} \int\limits_{K} \norm{u  (p)}_{S_{p}}^2\,\mathrm{d}\mu 
		=
		\frac{1}{l} \norm{u}_{2}^2,
	\end{align*}
	so that all in all
	\begin{equation*}
		\norm{u}_{1} 
		\geq
		\frac{1}{\sqrt{L\cdot l}}\norm{u}_{2}.\qedhere
	\end{equation*}
\end{proof}

\begin{prop}[{\cite[IV.2.2]{Wells:DiffAna}} - without proof]\label{prop:ctsExt}
	Every differential operator $D$ on $M$ has a continuous extension to an operator \mbox{$\mathrm{L}_{1}^2 (K;S)\to \mathrm{L}^2 (K;S)$} where $K\subseteq M$ is any compact subset.
\end{prop}

\begin{cor}\label{cor:SobInMaxdom}
	For $D$ a symmetric differential operator on $M$ and $K\subseteq M$ compact, $\mathrm{L}_{1}^2 (K;S)$ is contained in the domain of $D^*$.
\end{cor}

\begin{proof}
	For $u\in \mathrm{L}_{1}^2 (K;S)$, we need to show that there exists $C>0$ such that 
	\begin{equation*}
		\abs{\inner{u}{Dv}} \leq C \cdot \norm{v}_{2}
	\end{equation*}
	for all $v\in \Gamma^\infty_{c} (M;S)=\dom D$.
	Let $(u_{n})_{n}$ be a sequence in $\Gamma^\infty(K;S)$ converging to $u$ in $\norm{\,\cdot\,}_{1}$. By \Cref{prop:ctsExt}, $(Du_{n})_{n}$ converges in $\mathrm{L}^2 (M;S)$, so the $\norm{\,\cdot\,}_{2}$\=/norm of the sequence is bounded by some number $N$. For $0\neq v$, take some big enough $n$ such that $\norm{u - u_{n}}_{1} \leq \frac{\norm{v}_{2}}{c(\norm{D v}_{2} +1)}$ where $c>0$ is as in \Cref{lem:sobolev-norm-bigger}, and compute
	\begin{align*}
		\abs{\inner{u}{Dv}}
		&\leq
		\abs{\inner{u-u_{n}}{ Dv}}
		+
		\abs{\inner{u_{n}}{D v}}
		=
		\abs{\inner{u-u_{n}}{ Dv}}
		+
		\abs{\inner{Du_{n}}{ v}}\notag\\
		&\leq
		\norm{u-u_{n}}_{2}\norm{ Dv}_{2}
		+
		\norm{Du_{n}}_{2} \norm{ v}_{2}
		<
		(1+N) \norm{v}_{2}.\qedhere
	\end{align*}
\end{proof}

We will need the following \namecref{Rellich}s later, but we will not prove them here.
\begin{prop}[Rellich Lemma; {\cite[10.4.3]{HigRoe:KHom}}, {\cite[IV.1.2]{Wells:DiffAna}} - without proof]\label{Rellich}
	For $K\subseteq M$ compact, the inclusion $\mathrm{L}^2_{1} (K; S)\hookrightarrow \mathrm{L}^2 (K; S)$ is a compact operator.
\end{prop}

\begin{prop}[G{\aa}rding's Inequality; {\cite[10.4.4]{HigRoe:KHom}} -  without proof]\label{Garding}
	Suppose $M$ is compact. If $D$ is an elliptic differential operator on $M$, then there is a constant $c>0$ such that, for all $u\in \mathrm{L}^2_{1} (M; S)$,
	\[
	c \cdot \norm{u}_{1} \leq \norm{u}_{2} + \norm{Du}_{2}.
	\]
\end{prop}
As mentioned in the introduction, to be able to invoke G{\aa}rding's Inequality is the reason why we need to assume ellipticity of $D$ in our main theorem.


\subsection{Fourier Transforms and Normalizing functions}\label{subsec:NormalizingAndFT}

Most of the statements below can be found in~\cite[Chapters~8 and~9]{Folland:RA}.

\begin{notation}
	For~$f\colon \RR^n \to \mathbb{C}$ and~$x,y\in\RR^n$, let
	\[
		\bigl(\tau_{x} f\bigr) (y) := f(y-x)
		\quad\;\text{ and }\;\quad
		\tilde{f}(y) := f(-y).
	\]
For~$f\in \mathrm{L}^1 (\RR^n)$, its \textit{Fourier} and \textit{inverse Fourier Transform} are given by
	\begin{align}\label[definition]{eq:FT}
		\hat{f} (x) = \int\limits_{\RR^n} \mathsf{e}^{-2\pi\mathbf{i} x\cdot y} f(y) \,\mathrm{d}y 
		\;\text{ and }\;
		\check{f} (x) = \int\limits_{\RR^n}  \mathsf{e}^{2\pi\mathbf{i} x\cdot y} f(y) \,\mathrm{d}y .
	\end{align}
\end{notation}

If~$f, g\in\mathrm{L}^1$, then
	\begin{align}\label{eq:hatcanmove}
		\int\limits_{\RR^n} \hat{f}(x)g(x)\,\mathrm{d}x 
		=
		\int\limits_{\RR^n} {f}(x)\hat{g}(x)\,\mathrm{d}x .
	\end{align}
As a consequence, one can show that if~$f, \hat{f}$ are both~$\mathrm{L}^1$, then the \textit{inversion formula} holds: for almost every~$x\in \RR^n$, we have
	\[
		f (x) = \check{\hat{f}} (x) =
		\int\limits_{\RR^n} 
		\mathsf{e}^{2\pi\mathbf{i} x\cdot y}
		\hat{f}(y)
		\,\mathrm{d}y .
	\]

\begin{defn}
	The \textit{Schwartz space}~$\Schw$ consists of those smooth functions on~$\RR^n$ which have rapidly decaying derivatives. To be more precise, define for~$N\in\NN$ and a multi-index~$\alpha$,
	\begin{align}\label[definition]{def:SchwNorms}
		\norm{\phi}_{N,\alpha} 
		:=
		\sup_{x\in\RR^n}{(1+\norm{x})^N \abs{\partial^\alpha \phi(x)}}.
	\end{align}
	Then
	\[
		\Schw
		:=
		\left\{
			\phi\in C^\infty(\RR^n)
			\,\vert\,
			\text{for any } N\in\NN, \alpha \text{ multi-index}:
			\norm{\phi}_{N,\alpha} <\infty
		\right\}.
	\]
\end{defn}
When equipped with the seminorms given in \Cref{def:SchwNorms},~$\Schw$ becomes a Fr{\'e}chet space, cf.~\cite[8.2.~Proposition]{Folland:RA}. The Fourier Transform then maps~$\Schw$ continuously into itself and, because of the inversion formula, is hence an isomorphism of~$\Schw$ (cf.~\cite[8.28~Cor.]{Folland:RA}).
	
\begin{defn}
	A \textit{distribution}~$F$ is a functional on~$C_{c}^\infty (\RR^n)$. We will denote by $\iinner{F}{\phi}$ the value of~$F$ at the point \mbox{$\phi\in C_{c}^\infty (\RR^n)$,} and let~$\mathcal{D}'$ be the space of distributions. The \textit{support} of~$F$ is the complement of the maximal open subset~$U\subseteq\RR^n$ for which
	\begin{equation*}
		\iinner{F}{\phi}
		= 0
	\end{equation*}
	for all $\phi$ such that $\mathsf{supp}(\phi)\subseteq U$.
	A distribution~$F$ is \textit{tempered} if it extends continuously to all of~$\Schw$. As~$C_{c}^\infty (\RR^n)$ is dense in~$\Schw$ (cf.~\cite[9.9~Prop.]{Folland:RA}), the space of tempered distributions is the dual space~$\Schw'$ of~$\Schw$.
\end{defn}

\begin{example}\label{ex:locintgivesdist}
	If~$f\colon \RR^n \to \mathbb{C}$ is locally integrable (that is, integrable on compact sets), then it defines a distribution by
	\begin{align*}
		\iinner{f}{\phi}
		:=
		\int\limits_{\RR} f(x)\phi(x)\,\mathrm{d}x
	\end{align*}
	for $\phi\in C_{c}^\infty (\RR^n)$.
\end{example}

If~$\psi\in C_{c}^\infty (\RR^n)$, then
\[
	\int\limits_{\RR} \bigl(f\ast \psi\bigr)(x)\phi(x)\,\mathrm{d}x 
	=
	\int\limits_{\RR}\!\!
		\int\limits_{\RR}
			f(x) \psi(y-x)
			\phi(x)
		\,\mathrm{d}y 
	\,\mathrm{d}x 
	=
	\int\limits_{\RR} f(x)\bigl(\phi\ast \tilde{\psi}\bigr)(x)\,\mathrm{d}x ,
\]
and so the above \nameCref{ex:locintgivesdist} justifies the following \namecref{def:convForDist}:

\begin{defn}[{\cite[p.~285]{Folland:RA}}]\label{def:convForDist}
	If~$F\in\D'$ and~$\psi\in C_{c}^\infty (\RR^n)$, we define for $\phi\in C_{c}^\infty (\RR^n)$,
	\begin{align*}
		\iinner{F\ast\psi}{\phi}
		:=
		\iinner{F}{\phi\ast\tilde{\psi}}.
	\end{align*}
	One can show (see~\cite[9.3~Prop.]{Folland:RA}) that this distribution is actually given by integration against the function $F\ast\psi$ defined by
	\[
		F\ast\psi (x) := \iinner{F}{\tau_{x} \tilde{\psi}}.
	\]
\end{defn}

\begin{lem}\label{lem:ConvCmpctSupp}
	If~$F\in\D'$ has compact support and if~$\psi\in C_{c}^\infty (\RR^n)$, the function~$F\ast \psi$ is a smooth compactly supported function.
\end{lem}

\begin{proof} 
	Regarding smoothness, see~\cite[9.3a)~Prop.]{Folland:RA}. If we let~$A$ be the closure of
	$
		\mathsf{supp}(F)+\mathsf{supp}(\psi),
	$
	then~$A$ is compact by assumption. For~$x\notin A$, the function
	\[
		\tau_{x} \tilde{\psi}:\quad y \mapsto		{\psi} (x-y)
	\]
	is supported outside of~$\mathsf{supp}(F)$, so that
	\begin{align*}
		F\ast\psi (x)
		=
		\iinner{F}{\tau_{x} \tilde{\psi}}
		=
		0.\qedhere
	\end{align*}
\end{proof}

\begin{example}[special case of \Cref{ex:locintgivesdist}]\label{bddGivesTempDist}
	If~$f: \RR\to\mathbb{C}$ is measurable and bounded, then it defines a \emph{tempered} distribution by
	\begin{align*}
	 	\iinner{f}{\phi}
	 	:=
	 	\int\limits_{\RR} f(x)\phi(x)\,\mathrm{d}x
	\end{align*}
	for $\phi\in\Schw$.
	Indeed, since
	\[
		\sup_{x\in\RR}{(1+\norm{x})^2 \abs{\phi(x)}}
		=
		\norm{\phi}_{2,0}
		<\infty,
	\]
	we have
	\begin{align*}
		\int\limits_{\RR}
			\abs{
				f(x)\phi(x)
			}
		\,\mathrm{d}x 
		\leq
		\int\limits_{\RR}
			\supnorm{f}
			\frac{\norm{\phi}_{2,0}}{(1+\abs{x})^2}
		\,\mathrm{d}x 
		\leq
		\int\limits_{\RR}
			\frac{\supnorm{f}\norm{\phi}_{2,0}}{1+\abs{x}^2}
		\,\mathrm{d}x 
		=
		\supnorm{f}\norm{\phi}_{2,0}\pi < \infty,
	\end{align*}
	so the measurable function~$f\phi$ is integrable, and~$\iinner{f}{\phi}$ is well\-/defined and continuous.
\end{example}
The advantage of tempered distributions over other distributions is the following definition:
\begin{defn}
	If~$F$ is a tempered distribution, we define its \textit{Fourier} and \textit{inverse Fourier Transform} by
	\begin{align*}
		\iinner{\hat{F}}{\phi} := \iinner{{F}}{\hat{\phi}},
		\;\text{ and }\;
		\iinner{\check{F}}{\phi} := \iinner{{F}}{\check{\phi}}
	\end{align*}
	for $\phi\in\Schw.$
	Because of \Cref{eq:hatcanmove}, we see that, if~$F$ is integration against an~$\mathrm{L}^1$\-/function, then both~$\hat{F}$ and~$\check{F}$ agree with the definition given in \Cref{eq:FT}. Moreover, we again have the inversion formula~$\hat{\check{F}}=\check{\hat{F}}=F$.
\end{defn}

\begin{lem}\label{lem:pv-gives-temp-dist}
	Suppose we are given an even, integrable function~$h\colon \RR\to\mathbb{C}$. Then the assignment
	\begin{align*}
		\pv{\frac{h(t)}{t}}
		\colon 
		C_{c}^\infty (\RR)
		&\longrightarrow
		\mathbb{C},\\
		\varphi &\longmapsto
		\lim_{\epsilon \to 0^+}
		\left(
			\int\limits_{-\infty}^{-\epsilon}
				\frac{h(t)}{t} \varphi(t)
			\,\mathrm{d}t 
			+		
			\int\limits_{\epsilon}^{\infty}
				\frac{h(t)}{t} \varphi(t)
			\,\mathrm{d}t 
			\right),
	\end{align*}
	extends continuously to~$\Schw$. Furthermore, the Fourier Transform of this tempered distribution is given by integration against the \textup{(}well\-/defined\textup{)} function
	\[
		\zeta (x)
		:=
		\int\limits_{-\infty}^{\infty}
			\frac{\sin (tx)}{\mathbf{i} t} h\left(\frac{t}{2\pi}\right)
		\,\mathrm{d}t .
	\]
\end{lem}

\begin{proof}
	Let~$\varphi\in C_{c}^\infty (\RR)$. For any~$\epsilon >0$, the following two integrals exist since~$h$ is integrable, and are equal because~$h$ is even:
	\[
		\int\limits_{-\infty}^{-\epsilon}
			\frac{h(t)}{t}\varphi(0)
		\,\mathrm{d}t 
		=
		-
		\int\limits_{\epsilon}^{\infty}
			\frac{h(t)}{t}\varphi(0)
		\,\mathrm{d}t .
	\]
	Therefore, we may rewrite
	\begin{align*}
		\iinner{\pv{\frac{h(t)}{t} }}{\varphi}
		&=
		\lim_{\epsilon \to 0^+}
		\left(
			\int\limits_{-\infty}^{-\epsilon}
				\frac{\varphi(t)-\varphi(0)}{t} h(t)
			\,\mathrm{d}t 
			+		
			\int\limits_{\epsilon}^{\infty}
				\frac{\varphi(t)-\varphi(0)}{t} h(t)
			\,\mathrm{d}t 
		\right)\\
		&=
		\int\limits_{-\infty}^{\infty}
			\frac{\varphi(t)-\varphi(0)}{t} h(t)
		\,\mathrm{d}t ,
	\end{align*}
	where the last line holds because~$t\mapsto\frac{\varphi(t)-\varphi(0)}{t}$ can be smoothly extended at~$0$ by the value~$\varphi'(0)$ by L'H{\^o}pital. We therefore get
	\begin{align*}
		\abs{
			\iinner{\pv{\frac{h(t)}{t} }}{\varphi}
		}
		&\leq
		\int\limits_{-\infty}^{\infty}
			\abs{\frac{\varphi(t)-\varphi(0)}{t} h(t)}
		\,\mathrm{d}t 
		\leq
		\norm{h}_{\mathrm{L}^1}\cdot
			\sup_{t\in\RR}{\abs{\frac{\varphi(t)-\varphi(0)}{t}}}
		\\
		&\leq
		\norm{h}_{\mathrm{L}^1}\cdot
		\sup_{t\in\RR}{\abs{\varphi'(t)}}
	\end{align*}
	by the Mean Value Theorem.
	In particular, the value is finite for~$\varphi\in C_{c}^\infty (\RR)$ and in fact also for~$\varphi\in\Schw$. Moreover, given~$\varphi_{k}\in \Schw$ converging to~$0$, the above line means that
	\begin{align*}
		\abs{
					\iinner{\pv{\frac{h(t)}{t} }}{\varphi_{k}}
		}
		&\leq
		\norm{h}_{\mathrm{L}^1}\cdot
		\norm{
			\varphi_{k}
		}_{0,1} \stackrel{k\to\infty}{\longrightarrow} 0,
	\end{align*}
	so we have shown that our functional extends \emph{continuously} to~$\Schw$.
	
	Regarding~$\zeta$, first note that the (scaled) sinc function~$\RR^{\times}\ni t\mapsto \frac{\sin(tx)}{t}$ can be continuously extended at~$0$ by assigning it the value~$x$, and that it is bounded by~$\abs{x}$. Hence, since~$h$ is integrable, we see that~$\zeta(x)$ is actually a finite number,  so~$\zeta$ is well\-/defined. To check that~$\zeta$ is the Fourier transform, we can equivalently show that~$\check{\zeta} =	\pv{\frac{h(t)}{t} }$, so consider
	\begin{align*}
		\iinner{\check{\zeta}}{\varphi}
		=
		\iinner{\zeta}{\check{\varphi}}
		=
		\int\limits_{-\infty}^{\infty}
		 \zeta(x)
			\check{\varphi}(x)	
		\,\mathrm{d}x 
		=
		\int\limits_{-\infty}^{\infty}
		\left(
			\int\limits_{-\infty}^{\infty}
				\frac{\sin (tx)}{\mathbf{i} t} h\left(\frac{t}{2\pi}\right)
			\,\mathrm{d}t 
		\right)
		\check{\varphi}(x)
		\,\mathrm{d}x .
	\end{align*}
	As mentioned above,~$\abs{\frac{  \sin(tx)}{  \mathbf{i} t}} \leq \abs{x}$, so since~$h$ and~$x\check{\varphi}$ are integrable (the latter because~$\varphi\in\Schw$), we can use the Dominated Convergence Theorem to get
	\begin{align*}
		\iinner{\check{\zeta}}{\varphi}
		=
		\lim_{\epsilon \to 0^+}
			\int\limits_{-\infty}^{\infty}
			\left(
				\int\limits_{-\infty}^{-\epsilon}
					\frac{\sin (tx)}{\mathbf{i} t} h\left(\frac{t}{2\pi}\right)
				\,\mathrm{d}t 
				+
				\int\limits_{\epsilon}^{\infty}
					\frac{\sin (tx)}{\mathbf{i} t} h\left(\frac{t}{2\pi}\right)
				\,\mathrm{d}t 
			\right)
			\check{\varphi}(x)
			\,\mathrm{d}x 
		.
	\end{align*}
	Now again, for any~$\epsilon >0$, the following two integrals exist since~$h$ is integrable, and are equal because~$h$ is even:
	\[
		\int\limits_{-\infty}^{-\epsilon}
		\frac{\cos(xt)}{t} h\left(\frac{t}{2\pi}\right)
		\,\mathrm{d}t 
		=
		-
		\int\limits_{\epsilon}^{\infty}
		\frac{\cos(xt)}{t} h\left(\frac{t}{2\pi}\right)
		\,\mathrm{d}t .
	\]
	Therefore, with the previous computation, 
	\begin{align*}
	\iinner{\check{\zeta}}{\varphi}
	&=
	\lim_{\epsilon \to 0^+}
	\int\limits_{-\infty}^{\infty}
		\left(
			\int\limits_{-\infty}^{-\epsilon}
				\frac{-\mathsf{e}^{\mathbf{i} tx}}{t} h\left(\frac{t}{2\pi}\right)
			\,\mathrm{d}t 
			+
			\int\limits_{\epsilon}^{\infty}
				\frac{-\mathsf{e}^{\mathbf{i} tx}}{t} h\left(\frac{t}{2\pi}\right)
			\,\mathrm{d}t 
		\right)
	\check{\varphi}(x)
	\,\mathrm{d}x 
	\\
	&=
		\lim_{\epsilon \to 0^+}
			\int\limits_{-\infty}^{\infty}
				\left(
					\int\limits_{\epsilon}^{\infty}
						\mathsf{e}^{-2\pi\mathbf{i} tx}\frac{ h(t)}{t}
					\,\mathrm{d}t 
					+
					\int\limits_{-\infty}^{-\epsilon}
						\mathsf{e}^{-2\pi\mathbf{i} tx}\frac{ h(t)}{t}
					\,\mathrm{d}t 
				\right)
				\check{\varphi}(x)
			\,\mathrm{d}x .
	\end{align*}
	A standard use of Tonelli's and Fubini's Theorem shows that we can interchange the order of integration, so that
	\begin{align*}
		\iinner{\check{\zeta}}{\varphi}
		=
		\lim_{\epsilon \to 0^+}
		\left[
			\int\limits_{-\infty}^{-\epsilon}
				\left(
					\int\limits_{-\infty}^{\infty}
						\mathsf{e}^{-2\pi\mathbf{i} tx}
						\check{\varphi}(x)
					\,\mathrm{d}x 
				\right)
				\frac{ h(t)}{t}
			\,\mathrm{d}t 
			+
			\int\limits_{\epsilon}^{\infty}
				\left(
					\int\limits_{-\infty}^{\infty}
						\mathsf{e}^{-2\pi\mathbf{i} tx}
						\check{\varphi}(x)
					\,\mathrm{d}x 
				\right)
				\frac{ h(t)}{t}
				\,\mathrm{d}t 			
		\right].
	\end{align*}
	Since~$\varphi\in \Schw$, we know that the inversion formula holds: for almost every $t$, we have
	\[
		\varphi (t) = \hat{\check{\varphi}} (t) =
		\int\limits_{-\infty}^{\infty}
			\mathsf{e}^{-2\pi\mathbf{i} xt}
			\check{\varphi}(x)
		\,\mathrm{d}x.
	\]
	Therefore,
	\begin{equation*}
		\iinner{\check{\zeta}}{\varphi}
		=
		\lim_{\epsilon \to 0^+}
		\left[
		\int\limits_{-\infty}^{-\epsilon}
		\varphi (t)
		\frac{ h(t)}{t}
		\,\mathrm{d}t 
		+
		\int\limits_{\epsilon}^{\infty}
		\varphi (t)
		\frac{ h(t)}{t}
		\,\mathrm{d}t 			
		\right]
		=
		\iinner{\pv{\frac{h(t)}{t} }}{\varphi}.\qedhere
	\end{equation*}
\end{proof}

\begin{defn}\label{def:normalizing}
	A smooth function~$\chi\colon  \RR\to [-1,1]$ is a \textit{normalizing function} if
	\begin{enumerate}
		\item\label[property]{normalizing:odd} $\chi$ is odd,
		\item\label[property]{normalizing:positive} for $x>0$ we have $\chi(x) >0$, and
		\item\label[property]{normalizing:asymptotes} for $x\to \pm\infty$, we have $\chi(x) \to \pm 1$.
	\end{enumerate}
\end{defn}
\begin{lem}
	\label{ChiExists}
	For every~$\epsilon >0$, there exists a normalizing function~$\chi$ whose \textup{(}distributional\textup{)} Fourier transform is supported in~$(-\epsilon,  \epsilon)$.
\end{lem}

\begin{proof}\label{pf:ChiExists}
	We will follow the instructions in \cite[Exercise~10.9.3]{HigRoe:KHom}.
	
	Fix an even function~$g\in C_{c}^{\infty} (\RR,\RR)$ such that~$g\ast g(0) = \frac{1}{\pi}$. One could, for example, take a rescaled version of the function
	\[
		t\mapsto
		\left\{
		\begin{array}{ll}
		\exp(-\frac{1}{1-t^2}) & \mbox{if } \abs{t}<1, \\
		0 & \mbox{otherwise.}
		\end{array}
		\right.
	\]
	Let~$f:= g\ast g$, and define
	\[
	\chi(x)
	:=
	\int\limits_{-\infty}^{\infty} \frac{\sin(xt)}{t}f(t) \,\mathrm{d}t ,
	\]
	which is well-defined (see proof of \Cref{lem:pv-gives-temp-dist}), odd, and smooth.
	
	Now, recall that for any~$a>0$, the sinc function is the Fourier transform of a scaled characteristic function, namely
	\[
		\frac{\sin(2\pi at)}{\pi t} = \hat{\mathbb{1}}_{[-a,a]}(t).
	\]
	Using \cite[Lemma~8.25]{Folland:RA}, we can thus rewrite~$\chi$ for positive~$x$ as follows:
	\begin{align*}
		\chi(x)
		&=
		\pi \int\limits_{-\infty}^{\infty} \hat{\mathbb{1}}_{[-\frac{x}{2\pi},\frac{x}{2\pi}]}(t)f(t) \,\mathrm{d}t 
		=
		\pi \int\limits_{-\infty}^{\infty} {\mathbb{1}}_{[-\frac{x}{2\pi},\frac{x}{2\pi}]}(t)\hat{f}(t) \,\mathrm{d}t 
		\\
		&=
		\pi \int\limits_{-\frac{x}{2\pi}}^{\frac{x}{2\pi}}\hat{f}(t) \,\mathrm{d}t 
		=
		\pi \int\limits_{-\frac{x}{2\pi}}^{\frac{x}{2\pi}}\hat{g}(t)^2 \,\mathrm{d}t ,
	\end{align*}
	from which we see that~$\chi(x)\geq 0$. Moreover, since~$g$ is non\-/zero and smooth with compact support,~$\hat{g}$ does not vanish on any interval (cf.\ \cite[p.~293]{Folland:RA}). The above equality hence gives~$\chi(x) > 0$ for~$x>0$, so~$\chi$ satisfies \Cref{normalizing:positive} of normalizing functions. Furthermore,
	\begin{align*}
		\pi \int\limits_{-\frac{x}{2\pi}}^{\frac{x}{2\pi}}\hat{g}(t)^2 \,\mathrm{d}t 
		\leq
		\pi \int\limits_{-\infty}^{\infty}\hat{g}(t)^2 \,\mathrm{d}t 
		=
		\pi \norm{\hat{g}}_{2}^2
		\stackrel{(\ast)}{=}
		\pi \norm{{g}}_{2}^2
		=
		\pi f(0)
		=1,
	\end{align*}
	where~$(\ast)$ holds because of the Plancherel Theorem (see \cite[8.29]{Folland:RA}), so we have shown that~$\chi$ is indeed~$[-1,1]$\=/valued. Next, the Dominated Convergence Theorem allows us to compute
	\[
	\lim_{x\to \infty} \chi(x)
	=
	\lim_{x\to \infty}
	\int\limits_{}^{} \frac{\sin(t)}{t}f(\frac{t}{x}) \,\mathrm{d}t
	\stackrel{\text{DCT}}{=}
	\int\limits_{}^{} \frac{\sin(t)}{t}f(0) \,\mathrm{d}t 
	=1,	
	\]
	so we have shown \Cref{normalizing:asymptotes} of normalizing functions.

	From \Cref{lem:pv-gives-temp-dist}, we see that
	\[
		\check{\chi} = \pv{\frac{f(2\pi t)}{\mathbf{i} t}},
	\]
	so
	\[
		\iinner{\hat{\chi}}{\varphi}
		=
		\iinner{\check{\chi}}{\tilde{\varphi}}
		=
		\lim_{\epsilon \to 0^+}
		\left(
			\int\limits_{-\infty}^{-\epsilon}
				\frac{f(2\pi t)}{\mathbf{i} t} \varphi(-t)
			\,\mathrm{d}t 
			+		
			\int\limits_{\epsilon}^{\infty}
				\frac{f(2\pi t)}{\mathbf{i} t} \varphi(-t)
			\,\mathrm{d}t 
		\right).
	\]
	Thus, if~$\varphi$ has support disjoint from the support of~$t\mapsto f(- 2\pi t)=f( 2\pi t)$, then \mbox{$\iinner{\hat{\chi}}{\varphi}=0$.} In other words, the support of~$\hat{\chi}$ is contained in~$\frac{1}{2\pi}\mathsf{supp}(f)$, which is compact.
	
	Lastly, out of~$\chi$ with Fourier Transform supported in, say,~$(-b,b)$, we want to construct another normalizing function whose Fourier transform is supported in~$(-\epsilon, \epsilon)$. Let~$T(x) := \frac{\epsilon \, x}{b}$ and~$\chi_{2} := \chi \circ T$. As~$\epsilon,  a$ are positive, this is again a normalizing function, and we compute for~$\varphi\in\Schw$,
	\begin{align*}
		\iinner{\hat{\chi}_{2}}{\varphi}
		=
		\iinner{\chi_{2}}{\hat{{\varphi}}}
		&=
		\int\limits_{-\infty}^{\infty} \chi_{2}(x) \hat{\varphi}(x) \,\mathrm{d}x 
		=
		\int\limits_{-\infty}^{\infty} (\chi\circ T)(x) \hat{\varphi}(x) \,\mathrm{d}x \\
		&
		=
		\int\limits_{-\infty}^{\infty} \chi(x) (\hat{\varphi}\circ T^\mi) (x) (T^\mi)'(x) \,\mathrm{d}x .
	\end{align*}
	From \cite[Thm.~8.2b)]{Folland:RA} we know that
	\[
	(\hat{\varphi}\circ T^\mi) \cdot (T^\mi)'
	=
	(\varphi\circ T)\hat{}\,.
	\]
	If~$\varphi$ is now supported outside of~$(-\epsilon,  \epsilon)$, so that~$\varphi\circ T$ is supported outside of~$(-b,b)$, then the above computations yield
	\begin{align*}
		\iinner{\hat{\chi}_{2}}{\varphi}
		&=
		\int\limits_{-\infty}^{\infty} \chi(x) (\varphi\circ T)\hat{}(x)  \,\mathrm{d}x 
		=
		\iinner{\hat{\chi}}{\varphi\circ T}
		=0.
	\end{align*}
	This proves that~$\hat{\chi}_{2}$ is supported in~$(-\epsilon, \epsilon)$.
\end{proof}

\begin{lem}[{\cite[Prop.\ 10.3.5]{HigRoe:KHom}}]\label{lem:10.3.5}
	If~$D$ is an essentially self\-/adjoint  differential operator on~$M$ and~$\psi$ a bounded Borel function on~$\RR$ whose Fourier transform has compact support, then for all~$u, v\in \Gamma^\infty_{c} (M;S)$, we have
	\[
	\biginner{\psi(D){u}}{{v}} 
	=
	\Bigiinner{\hat{\psi}}{s\mapsto \inner{\mathsf{e}^{2\pi isD}{u}}{{v}}}
	.
	\]
\end{lem}

\begin{proof}
	We follow the idea given in \cite[Prop.\ 10.3.5]{HigRoe:KHom}. If we first take~$\psi_{1} \in \Schw$, then~$\psi_{1} = \check{\hat{\psi}}_{1}\,$, so that 
	\[
		\inner{\psi_{1}(D) u}{v}
		=
		\inner{
			\left(
				\int \mathsf{e}^{2\pi\mathbf{i} s D} \hat{\psi}_{1}(s) \,\mathrm{d}s 
			\right) u}
		{v}
		=
		\int
		\inner{
			 \mathsf{e}^{2\pi\mathbf{i} s D}u  }
		{v}
		\hat{\psi}_{1}(s) \,\mathrm{d}s .
	\]
	Since for functions in~$\mathrm{L}^1(\RR)$, the classical Fourier transform coincides with the distributional Fourier transform
	, the above equation can be rewritten as
	\begin{align*}
		\inner{\psi_{1}(D) u}{v}
		=
		\iinner{\hat{\psi}_{1}}{g},
	\end{align*}
	where~$g(s):=\inner{\mathsf{e}^{2\pi\mathbf{i} s D}u }{v}$, which was to be shown. Using the inversion formula for~$\psi_{2}=\hat{\psi}_{1}\in \Schw$ once more, we could also write this as
	\begin{equation}\label{ClaimForSchwartz}
		\inner{\check{\psi}_{2}(D) u}{v}
		=
		\iinner{{\psi}_{2}}{g}
	\end{equation}
	for~$\psi_{2} \in\Schw$ arbitrary
	Now let us take a general~$\psi$ as specified in the \namecref{lem:10.3.5}. As explained in \Cref{bddGivesTempDist},~$\psi$ gives rise to a tempered distribution, denoted by~$F$ for now. In particular, it makes sense to speak of its Fourier transform. Fix some \mbox{$\phi\in C_{c}^{\infty} (\RR,\RR)$} with \mbox{$\int \phi(x) \,\mathrm{d}x  = 1$,} and define \mbox{$\phi_{t} (x) := \frac{1}{t} \phi(\frac{x}{t})$.} Since we have assumed~$\hat{F}$ to have compact support,~$\hat{F}\ast \phi_{t} \in C_{c}^{\infty}(\RR)$ by \Cref{lem:ConvCmpctSupp}, so that \Cref{ClaimForSchwartz} implies
	\begin{align}\label{intermediatestep}
		\inner{(\hat{F}\ast \phi_{t})\check{}\, (D) u}{v}
		=
		\iinner{\hat{F}\ast \phi_{t}}{g}.
	\end{align}
	
	
	If we can now show that
	\begin{enumerate} 
		\item\label{No0}~$(\hat{F}\ast \phi_{t})\check{}\,=\psi\cdot\check{{\phi}}_{t}$,
		\item\label{No1}~$\underset{t\rightarrow 0}{\mathrm{lim}} \inner{(\psi \cdot \check{{\phi}}_{t}) (D) u}{v} = \inner{\psi (D) u}{v}$,  and
		\item\label{No2}~$\underset{t\rightarrow 0}{\mathrm{lim}} \iinner{\hat{F}\ast \phi_{t}}{g} = \iinner{\hat{F}}{g}$,
	\end{enumerate}
	then 
	\begin{align*}
		\inner{\psi (D) u}{v}
		=
		\underset{t\rightarrow 0}{\mathrm{lim}} \inner{(\hat{F}\ast \phi_{t})\check{}\, (D) u}{v}
		\stackrel{\text{(\ref{intermediatestep})}}{=}
		\underset{t\rightarrow 0}{\mathrm{lim}}	\iinner{\hat{F}\ast \phi_{t}}{g}
		=
		\iinner{\hat{F}}{g},
	\end{align*}
	so we would be done.
	\begin{description}
				\item[ad \textup{(\ref{No0})}] By virtue of \cite[p.~283]{Folland:RA}, it suffices to check that the functions induce the same distribution: we recall that~$\tilde{\phi} (x) = {\phi} (-x)$, and compute for~$f\in C_{c}^\infty$,
		\begin{align*}
		\qquad\iinner{(\hat{F}\ast \phi_{t})\check{}\,}{f}
		&=
		\iinner{\hat{F}\ast \phi_{t}}{\check{f}}
		=
		\iinner{\hat{F}}{\check{f}\ast\tilde{\phi}_{t}}
		=
		\iinner{F}{\left(\check{f}\ast\tilde{\phi}_{t}\right)\hat{}\,}\\
				&
		=
		\iinner{F}{\hat{\check{f}}\cdot \hat{\tilde{\phi}}_{t} }=
		\int\limits_{-\infty}^{\infty} \psi (x) f(x) \hat{\tilde{\phi}}_{t}(x) \,\mathrm{d}x 
		=
		\iinner{\psi\cdot\check{{\phi}}_{t}}{f}.
		\end{align*}
		
		\item[ad \textup{(\ref{No1})}] Using \Cref{SR:FctCalcConvergence} of Functional Calculus, it is sufficient to show that $\left\{\supnorm{\psi\,\cdot \check{\phi}_{t}}\right\}_{t}$ is bounded and that~$\psi\cdot \check{\phi}_{t}$ converges pointwise to~$\psi$: first of all, \mbox{$\check{\phi}_{t}(x) = \hat{\phi}(-tx)$} implies
		\[
			\qquad\supnorm{\psi\cdot \check{\phi}_{t}}
			\leq
			\supnorm{\psi} \cdot \supnorm{\hat{\phi}} < \infty.
		\]
		Secondly, suppose~$\mathsf{supp} (\phi) \subseteq [-a,a]$, and take~$h\in C_{c}^\infty$ such that \mbox{$h_{|[-a,a]} \equiv 1$,} so that \mbox{$\phi_{t} =\phi_{t}\cdot h$} for~$t\leq 1$. Since~$\phi_{t}\longrightarrow\delta$ in~$\mathcal{D}'$ for~$t\longrightarrow 0$ by \cite[Prop.~9.1]{Folland:RA}, we let for~$y\in \RR$, 
		\mbox{$f_{y} (x) := \mathsf{e}^{2\pi\mathbf{i} xy} h(x)\in C_{c}^\infty$} 
		and get
		\[
			\qquad\check{\phi}_{t} (y)
			=
			\int\limits_{-\infty}^{\infty}  \mathsf{e}^{2\pi\mathbf{i} xy}\phi_{t} (x) h(x)\,\mathrm{d}x 
			=
			\iinner{\phi_{t}}{f_{y}}\stackrel{t\to 0}{\longrightarrow}\iinner{\delta}{f_{y}} = f_{y}(0) = 1.
		\]
		\item[ad \textup{(\ref{No2})}] Recall that we defined 	\mbox{$g(s):=\inner{\mathsf{e}^{2\pi\mathbf{i} s D}u }{v}$} for fixed compactly supported sections~$u, v$. Since~$\partial^m g$ is bounded, it follows from \cite[Thm.~8.14(c)]{Folland:RA} that 
		\begin{equation*}
			\qquad\partial^m (g\ast \tilde{\phi}_{t}) = (\partial^m g) \ast \tilde{\phi}_{t} \stackrel{t\to 0}{\longrightarrow} \partial^m g
		\end{equation*}
		uniformly on compact sets.
		 Let us take~$h\in C_{c}^\infty$ such that~$0\leq h \leq 1$ and~$h_{|\mathsf{supp}(\hat{F})}\equiv 1$, where we use that~$\hat{F}$ is compactly supported. As each~$\partial^i h$ has compact support,  we get
		\[
			\qquad\supnorm{
				(\partial^i h)
				\cdot
				\bigl(
					\partial^m (g\ast \tilde{\phi}_{t}) - \partial^m g
				\bigr)
			}
			\stackrel{t\to 0}{\longrightarrow}
			0.
		\]
		It follows by the product rule that, for any~$k$, 
	\begin{align*}
		\qquad\supnorm{\partial^k \bigl( h \cdot [ g\ast \tilde{\phi}_{t} ]\bigr)
		-
		\partial^k \bigl( h \cdot g \bigr)} \stackrel{t\to 0}{\longrightarrow}
		0,
	\end{align*}				
	that is,~$h \cdot [ g\ast \tilde{\phi}_{t} ] \stackrel{t\to 0}{\longrightarrow} h\cdot g$ in~$C^\infty (\RR)$. As~$\hat{F}$ has compact support, it is in the dual space of~$C^\infty (\RR)$. Since~$h_{|\mathsf{supp}(\hat{F})}\equiv 1$, we therefore get
		\begin{align*}
			\qquad\iinner{\hat{F}\ast {\phi}_{t}}{ g }
			=
			\iinner{\hat{F}}{ g\ast \tilde{\phi}_{t} }
			=
			\iinner{\hat{F}}{h \cdot [ g\ast \tilde{\phi}_{t} ]}
			\stackrel{t\to 0}{\longrightarrow}
			\iinner{\hat{F}}{h\cdot g}
			=
			\iinner{\hat{F}}{g},
		\end{align*}
	which finishes our proof.\qedhere
	\end{description}
\end{proof}


\section{The Main Theorem}\label{sec:Thm-main}

\begin{thm}[special case of {\cite[Thm.~10.6.5]{HigRoe:KHom}}]\label{thm:main}
	Let $D$ be a symmetric	elliptic differential operator on a smooth and compact manifold $M$. Let $\mathcal{H} := \mathrm{L}^2 (M; S)$ and let $\M{}$ be the representation of $C(M)$ on $\mathcal{H}$ by multiplication. For $\chi$ a normalizing function and $F:= \chi(D)$, the triple $(\M{}, \mathcal{H}, F)$ is a Fredholm module. Moreover, its class in $\Kth^{1} ( C(M))$	does not depend on the choice of $\chi$ and can hence be denoted by $[D]$.
\end{thm}

This is the \namecref{thm:main} we are set out to prove. As a first step, let us see that we have functional calculus at our disposal, so that $\chi(D)$ makes sense.

\begin{prop}[{\cite[Lemma~10.2.5]{HigRoe:KHom}}]\label{HigRoe:10.2.5}\label{cor:ess.sa}
	Let $D$ be a symmetric differential operator on a smooth manifold $M$ and let \mbox{$u\in \mathrm{L}^2 (M; S)$} be compactly supported. Then $u\in \dom \overline{D}$ if and only if $u\in \dom D^*$. In particular, if $M$ is compact, then $D$ is essentially self\-/adjoint.
\end{prop}

To prove \Cref{HigRoe:10.2.5}, we need the following two lemmas.
\begin{lem}[{\cite[Lemma~1.8.1]{HigRoe:KHom}} - without proof]\label{lem:minDom}
	If $D$ is a closable unbounded operator, then $u\in \dom \overline{D}$ if and only if there exists a sequence $\{u_{j}\}_{j}$ in $\dom D$ such that $u_{j} \to u$ and $\{\norm{D u_{j}} \}_{j}$  is bounded.
\end{lem}

\begin{lem}[{\cite[Exercise~10.9.1]{HigRoe:KHom}}]\label{Fmolli}	
	For $K\subseteq M$ compact, there exist for sufficiently small $\epsilon > t > 0$, operators $F_{t}\colon  \mathrm{L}^2 (K; S) \to \mathrm{L}^2 (M; S)$ which satisfy
	\begin{enumerate}[label=\textup{(\arabic*)}]
		\item\label[property]{Ft:norm} $\norm{F_{t}}\leq C$ for some constant $C$ and all $t$,
		\item\label[property]{Ft:1} $\forall  u\in \mathrm{L}^2 (K; S): \, \lim_{t\to 0} F_{t} u = u$ in $\mathrm{L}^2 (M; S)$,
		\item\label[property]{Ft:SmoothCpct} $\forall  u\in \mathrm{L}^2 (K; S):$ $F_{t} u$ is smooth with compact support, and
		\item\label[property]{Ft:bdd} for any differential operator $D$ on $M$, $\comm{D}{F_{t}}$ extends to a bounded operator \linebreak[4]\mbox{$\mathrm{L}^2 (K; S) \to \mathrm{L}^2 (M; S)$,} and its norm is bounded independent of $t$.
	\end{enumerate}
\end{lem}

We remark that the constant in \Cref{Ft:norm} is usually supposed to be $1$, but $C$ is good enough for us.
For a proof of the existence of these so\-/called \textit{Friedrichs' mollifiers}, see the appendix on p.~\pageref{pf:Fmolli}.

\begin{proof}[Proof of \Cref{HigRoe:10.2.5}.]
		Since the minimal domain of $D$ is always contained in the maximal domain, let us take $u\in \dom D^*$ with compact support (pick any representative). According to \Cref{lem:minDom}, we need to find a sequence of $v_{n}$ in $\dom D$ which converges to $u$ in the Hilbert space and such that $\{ \norm{Dv_{n}}\}_{n}$ is bounded.
		Let us take $F_{t}$ as in \Cref{Fmolli} for $K:= \mathsf{supp}(u)$,  let $t_{n}$ be a sequence converging to $0$, and let $v_{n} := F_{t_{n}} u$. Since $v_{n} \in \Gamma^\infty_{c} (M;S)$ by \Cref{Ft:SmoothCpct} of the mollifiers, it is in the domain of $D$, and by \Cref{Ft:1}, $v_{n} \to u$ in $\mathrm{L}^2 (M; S)$.  It remains to see why the sequence $D(v_{n})$ is bounded:
				
		By \Cref{AdjDPresveresSupp},  $D^* u$ is in $\mathrm{L}^2 (K; S)$, so $F_{t} (D^* u)$ makes sense. Moreover, by \Cref{Ft:bdd} of the mollifiers, we also have that $\comm{D}{F_{t}} u$ has a well-defined meaning. All in all, we can therefore write
		\[
			D(v_{n}) = D^*(F_{t_{n}} u) = F_{t_{n}} (D^* u) + \comm{D}{F_{t_{n}}} u.
		\]
		Because of \Cref{Ft:norm} and \Cref{Ft:bdd} of $\{F_{t}\}_{t}$, there exists $C>0$ such that for all $t$, we have $\norm{F_{t}}, \norm{ \comm{D}{F_{t}}} < C$. Hence 
		\[
			\norm{D(v_{n})}
			\leq \norm{F_{t_{n}} (D^* u)} + \norm{[D, F_{t_{n}}] u}
			\leq C \cdot(\norm{D^* u} + \norm{u}),
		\]
		so the sequence is indeed bounded.
	\end{proof}

The next \namecref{prop:ctsFuncCalcCpct} will show that the class $[D]$ does not depend on the choice of normalizing function $\chi$, and that $\chi(D)^2-1$ is compact.

\begin{prop}[{\cite[Prop.~10.4.5, Lemma~10.6.3]{HigRoe:KHom}}]\label{prop:ctsFuncCalcCpct}\label{lem:CpctPert}
	If $D$ is a symmetric elliptic differential operator on a compact manifold $M$, and $\varphi\in C_{0} (\RR)$, then \linebreak[4]\mbox{$\varphi(D)\colon  \mathrm{L}^2 (M; S)\to \mathrm{L}^2 (M; S)$} is a compact operator. In particular, if $\chi_{1}, \chi_{2}$ are normalizing functions, then the operators $\chi_{1}(D)$ and $\chi_{2} (D)$
		differ only by a compact operator.
\end{prop}

\begin{proof}
	We first want to show that  \mbox{$\dom \overline{D} = \mathrm{L}_{1}^2 (M; S)$} by showing the following containments:
	\[
		\dom \overline{D}\subseteq \mathrm{L}_{1}^2 (M; S)\subseteq\dom D^* = \dom \overline{D}.
	\]
	By \Cref{cor:ess.sa}, our symmetric operator is essentially self\-/adjoint (which explains the equality on the right), and \Cref{cor:SobInMaxdom} gives us \mbox{$\mathrm{L}_{1}^2 (M; S)\subseteq\dom D^*$.} Now suppose \mbox{$u\in \dom \overline{D}$,} that is, \mbox{$(u, \overline{D}u)\in \Gamma(\overline{D}) = \overline{\Gamma({D})}$.} This means there is a sequence \mbox{$(u_{j})_{j} \in \dom D$} such that  \mbox{$u_{j} \longrightarrow u$} and \mbox{$Du_{j} \longrightarrow \overline{D}u$} in \mbox{$\mathrm{L}^2 (M; S)$.} In particular,	\mbox{$(u_{j})_{j}$} is Cauchy in \mbox{$\mathrm{L}^2 (M; S)$,} so \hyperref[Garding]{G{\aa}rding's inequality} implies that \mbox{$(u_{j})_{j}$} is also Cauchy with respect to \mbox{$\norm{\,\cdot\,}_{1}$} (remember that $M$ is assumed compact). As \mbox{$\mathrm{L}_{1}^2 (M; S)$} is (by definition) complete with respect to this norm, \mbox{$(u_{j})_{j}$} thus has a $\norm{\,\cdot\,}_{1}$\-/limit in \mbox{$\mathrm{L}_{1}^2 (M; S)$.} The \hyperref[Rellich]{Rellich lemma}, for example, shows that this limit must coincide with $u$, so we have shown \mbox{$u\in \mathrm{L}_{1}^2 (M; S)$.} All in all, \mbox{$\dom \overline{D} = \mathrm{L}_{1}^2 (M; S)$.}

	Now let us focus on the function $\psi (x) = ({\mathbf{i}+x})^\mi$. Since the domain of $D$ is dense and $\overline{D}$ is self\-/adjoint, \Cref{SR:ResSurjective} implies that $\mathbf{i} + \overline{D}$ has full range.
	Thus, for every $u\in \mathrm{L}_{1}^2 (M; S)$, there exists $v\in \dom \overline{D} = \mathrm{L}_{1}^2 (M; S)$ such that $(\mathbf{i} +\overline{D})v = u$. Since $\overline{D}$ is self\-/adjoint, we know that 
	$$\norm{(\mathbf{i} +\overline{D})v}^2 = \norm{v}^2 + \norm{\overline{D}v}^2,$$
	see \Cref{normDpmi=normDmpi}.
	Hence it follows from \hyperref[Garding]{G{\aa}rding's inequality} and the properties of \hyperref[FuncCalc]{Functional Calculus} that, for some $c>0$, 
	\[
		c\cdot\norm{\psi(\overline{D}) u }_{1}
		=
		c\cdot\norm{v}_{1}
		\leq
		\norm{v} + \norm{\overline{D}v}
		\leq
		\sqrt{2}\norm{(\mathbf{i} +\overline{D})v}
		=
		\sqrt{2}\norm{u}.
	\]
	In other words, $\psi(\overline{D})$ is a bounded operator $\mathrm{L}^2 (M; S)\to\mathrm{L}_{1}^2 (M; S)$, and thus by the \hyperref[Rellich]{Rellich lemma}, it is a compact operator  $\mathrm{L}^2 (M; S)\to\mathrm{L}^2 (M; S)$.
	Lastly,	if we take an arbitrary $\varphi\in C_{0} (\RR)$, then for any $\epsilon >0$, there are finitely many $a_{i,j}\in\mathbb{C}$ such that
	\[
		\supnorm{
			\varphi
			-
			\sum_{i,j=0}^{m} a_{i,j} \psi^i \overline{\psi}{}^j
		}
		< \epsilon, 
	\]
	because $\psi$ generates $C_{0} (\RR)$ as a $C^*$\=/algebra. By \Cref{FctCalc-norm} of Functional Calculus, we get for \mbox{$f:= 	\varphi	-	\sum_{i,j=0}^{m} a_{i,j} \psi^i \overline{\psi}{}^j$} that
	\[
		\norm{f(\overline{D})}
		=
		\supnorm{f}
		< \epsilon .
	\]
	This means that the operator $\varphi(\overline{D})$ is approximated by compact operators 	and is hence itself compact.
\end{proof}
The remaining work before the proof of \Cref{thm:main} on page~\pageref{pf:thm-main} will culminate in \Cref{lem:CommutatorCpct}, which says that $\comm{\chi(D)}{\M{f}}$ is compact for $f\in C(M)$.
\begin{prop}[{\cite[Prop.~10.3.1]{HigRoe:KHom}}]\label{10.3.1}
	If $D$ is an essentially self\-/adjoint differential operator on $M$, and if $W$ is an open neighborhood of a compact set $K\subseteq M$, then there exists $\epsilon >0$ such that
	\[
		\forall  \abs{s} < \epsilon, 
		\forall  u\in \mathrm{L}^2 (K;S)
		:\quad
		\mathsf{supp}
		\left(
			\mathsf{e}^{\mathbf{i} s D} u
		\right)
		\subseteq W.
	\]
\end{prop}
\begin{proof}
	We will follow the proof given in \cite{HigRoe:KHom}. 
	Let~$g\in C^\infty_{c} (M, [0,1])$ be such that
	\[
		g|_{K} \equiv 1
		\quad\;\text{ and }\;\quad
		g|_{M\setminus W} \equiv 0.
	\]
	Pick~$f\in C^\infty (\RR,[0,1])$ non\-/decreasing such that
	\[
		\text{ for } t<1:\, f(t)<1,
				\quad\;\text{ and }\;\quad
		 \text{ for } t\geq 1:\,	f(t)=1.
	\]
	We have shown in \Cref{commDMisbdd} that~$\comm{D}{\M{g}}$ is bounded (even on all of~$\mathrm{L}^2 (M;S)$ since~$g$ is compactly supported), so let \mbox{$c > \norm{\comm{D}{\M{g}}}$.} We use this to define for~$s\in\RR^+$ and~$p\in M$:
	\[
		h_{s} (p) := f\bigl( g(p) + cs\bigr),
		\quad
		L_{s} := \{p\in M\,\vert\, h_{s} (p) =1\}.
	\]
	We will deal with positive~$s$ only; for negative~$s$, do the same construction for~$-D$.
	\begin{claim}
		If~$t\leq s$, then~$L_{t} \subseteq L_{s}$.
	\end{claim}
\begin{Claimproof}
		An element~$p$ is in~$L_{t}$ exactly if~$f\bigl( g(p) + ct\bigr) = 1$. By choice of~$f$, this means~$g(p) + ct \geq 1$. As~$s\geq t$ and~$c$ is positive, this implies~$g(p) + cs \geq 1$ also, hence \mbox{$f\bigl( g(p) + cs\bigr) = 1$.} Therefore,~$p\in L_{s}$.
\end{Claimproof}

	\begin{claim}\label{claim:Lvlsetcont}
		For~$0\leq s< \frac{1}{c}$, we have~$K\subseteq L_{0} \subseteq L_{s} \subseteq W$.
	\end{claim}
	\begin{Claimproof}
		For the first inclusion, use~$g|_{K} \equiv 1$ to see \mbox{$f( g(p)) = 1$} for any \mbox{$p\in K$} by choice of~$f$. The second inclusion follows from the above computation. For the last inclusion, recall that, if~$p\notin W$, then~$g(p) = 0$ by choice of~$g$. Since~$cs < 1$ by choice of~$s$, we therefore have~$h_{s} (p) = f( cs ) < 1$ by choice of~$f$.
	\end{Claimproof}
	
	Let us write~$\dot{h}_{s}$ to denote
	\begin{equation*}
		\dot{h}_{s} (p) := \partial_{s} \bigl( s\mapsto h_{s} (p) \bigr)_{|s}
		=
		c f' \bigl( g(p) + cs\bigr)
	\end{equation*}
	for $p\in M$.	
	Since~$c$ is positive and~$f$ is non\-/decreasing, we have~$\dot{h}_{s} (p) \geq 0$ for all~$s$ and~$p$.
	
	\begin{claim}\label{claim:CommutatorRelation}
		$\comm{D}{\M{h_{s}}} = \frac{1}{c} \M{\dot{h}_{s}}\comm{D}{\M{g}}.$
	\end{claim}
\begin{Claimproof}
			For~$D$ locally as in \Cref{DiffOpLook}, we have shown in \Cref{eq:commDM} that
			\[
			\qquad\bigl(\comm{D}{\M{h_{s}}} u\bigr)(p)
			=
			\sum_{j=1}^{n}
			\partial_{j} \bigl( h_{s} \circ\varphi^\mi\bigr)_{|\varphi (p)} \cdot 
			\Psi^\mi
			\left(
			p,
			A^j (p) 
			\cdot  (\psi  \circ u (p))
			\right),
			\]
			and similarly,
			\[
			\qquad\left(
			\frac{1}{c} \M{\dot{h}_{s}}\comm{D}{\M{g}} u
			\right) (p)
			=
			\frac{1}{c} {\dot{h}_{s}} (p)
			\sum_{j=1}^{n}
			\partial_{j} \bigl( g \circ\varphi^\mi\bigr)_{|\varphi (p)}\!\! \cdot \!
			\Psi^\mi
			\left(
			p,
			A^j (p) 
			\!\cdot\!  (\psi  \circ u (p))
			\right).
			\]
			If we write \mbox{$h_{s} = f\circ k_{s}$} where \mbox{$k_{s}(p) := g(p) + cs$,} then the chain rule gives
			\begin{align}
				\qquad\partial_{j} \bigl( h_{s} \circ\varphi^\mi\bigr)_{|\varphi (p)} \!
				=
				f' \bigl(k_{s}(p)\bigr)\,
				\partial_{j} \bigl(k_{s} \circ\varphi^\mi\bigr)_{|\varphi (p)} \!
				=
				\frac{1}{c}
				{\dot{h}_{s}} (p)\,
				\partial_{j} \bigl(g \circ\varphi^\mi\bigr)_{|\varphi (p)}
				\notag
			\end{align}
			for each~$1\leq j\leq n$, which implies the claim.
		\end{Claimproof}

	Because of \Cref{claim:CommutatorRelation}, we have
	\[
		\M{\dot{h}_{s}}\!-\mathbf{i} \comm{D}{\M{h_{s}}}
		=
		\frac{1}{c}\M{\dot{h}_{s}}
		\left(c-\mathbf{i} \comm{D}{\M{g}}\right).
	\]
	By choice of~$c$, we see that
	\[
		c\cdot 1 \geq \norm{\mathbf{i} \comm{D}{\M{g}}}\cdot 1 \geq \mathbf{i} \comm{D}{\M{g}},
	\]
	so~$c-\mathbf{i} \comm{D}{\M{g}} \geq 0.$	By \Cref{commDMisbdd},~$\comm{D}{\M{g}}$ is a multiplication operator, so it commutes with~$\M{\dot{h}_{s}}$. As~$\dot{h}_{s}$ is non\-/negative, we have therefore shown that 
	\begin{align}\label{ThingPositive}
		\M{\dot{h}_{s}}\!-\mathbf{i} \comm{D}{\M{h_{s}}}
		\geq 0.
	\end{align}
	Since it suffices to prove the \namecref{10.3.1} for~$u\in \Gamma^\infty (K;S)$, fix such~$u$ and define~\mbox{$u_{s} := \mathsf{e}^{\mathbf{i} s D} u$.} Since~$(\partial_{s} u_{s})_{|s} = \mathbf{i} D u_{s}$, we have
	\begin{align}
		\partial_{s} \inner{h_{s} \cdot u_{s}}{u_{s}}_{|s}
		&=
		\inner{\partial_{s} (h_{s}\cdot u_{s})_{|s}}{u_{s}}
		+
		\inner{h_{s}\cdot u_{s}}{(\partial_{s} u_{s})_{|s}}\notag\\
		&=
		\inner{\dot{h}_{s}\cdot u_{s} + h_{s} \cdot \mathbf{i} D u_{s}}{u_{s}}
		+
		\inner{h_{s}\cdot u_{s}}{\mathbf{i} D u_{s}}\notag\\
		&=
		\inner{\dot{h}_{s}\cdot u_{s} +  \mathbf{i} h_{s} \cdot D u_{s}}{u_{s}} - \inner{\mathbf{i}  D (h_{s}\cdot u_{s})}{u_{s}} \tag*{as~$D\subseteq D^*$}\\
		&=
		\inner{
			\left(
				\M{\dot{h}_{s}} -  \mathbf{i} \comm{D}{\M{h_{s}}}
			\right)
				u_{s}}{u_{s}}
		\geq 0
		\tag*{by \Cref{ThingPositive}.}
	\end{align}
	This means that~$\inner{h_{s} \cdot u_{s}}{u_{s}}$ is an increasing function, and in particular for~$s\geq 0$
	\begin{align*}
		\inner{h_{s} \cdot u_{s}}{u_{s}} \geq \inner{h_{0} \cdot u_{0}}{u_{0}} = \inner{h_{0} \cdot u}{u}
		\stackrel{(*)}{=}
		\inner{u}{u} = \inner{u_{s}}{u_{s}},
	\end{align*}
	where~$(*)$ holds because~$h_{0} = f\circ g$ is~$1$ on~$K\supseteq\mathsf{supp}(u)$, and the last equality comes from~$\mathsf{e}^{\mathbf{i} s D}$ being a unitary. Since~$1\geq h_{s} \geq 0$, this means
	\begin{align*}
		\norm{u_{s}}^2_{2} \geq \norm{\sqrt{h_{s}} u_{s}}^2_{2} = \inner{h_{s} \cdot u_{s}}{u_{s}} 
		\geq \inner{u_{s}}{u_{s}} = \norm{u_{s}}^2_{2}.
	\end{align*}
	Therefore,
	\begin{align*}
		\int\limits_M \norm{u_{s}(p)}^2_{S_{p}}\,\mathrm{d}\mu  = \int\limits_M \norm{\sqrt{h_{s}}(p) u_{s}(p)}^2_{S_{p}}\,\mathrm{d}\mu .
	\end{align*}
	Again, since~$1\geq h_{s} \geq 0$, we have~$\norm{u_{s}(p)}^2_{S_{p}}\geq \norm{\sqrt{h_{s}}(p) u_{s}(p)}^2_{S_{p}}$, and hence the equality of integrals implies
	\[
		\Sinner{u_{s}(p)}{u_{s}(p)}{p}\!
		=
		\Sinner{\sqrt{h_{s}}(p) \cdot u_{s}(p)}{\sqrt{h_{s}}(p) \cdot u_{s}(p)}{p}\!\!\!\!\!\!,
		\quad
		\text{so }
		\norm{\sqrt{1-h_{s}(p)}u_{s}(p)}_{S_{p}}\!\!=0.
	\]
	(This equality is actually true not only almost everywhere but for all~$p\in M$ since we are dealing with smooth functions.) This implies that~$h_{s} u_{s} = u_{s}$. In particular,~$\mathsf{supp}(u_{s})$ has to be contained in the set on which~$h_{s}$ is~$1$, that is,
	\begin{align*}
		\mathsf{supp}\left(\mathsf{e}^{\mathbf{i} s D} u\right)
		=
		\mathsf{supp}(u_{s}) \subseteq L_{s} \stackrel{}{\subseteq} W
	\end{align*}
	for~$s< \frac{1}{c}$ by \Cref{claim:Lvlsetcont}.	
	This finishes the proof of \Cref{10.3.1}.
\end{proof}
	
\setcounter{claim}{0}
\begin{cor}[{\cite[Cor.~10.3.3]{HigRoe:KHom}}]\label{lem:10.3.3}
	Let $D$ be an essentially self\-/adjoint differential operator on a manifold $M$. Let $f_{1}, f_{2}$ be bounded functions on $M$ with disjoint supports, and suppose $\mathsf{supp}(f_{2})$ is compact. Then there exists $\epsilon >0$ such that
	\[
		\forall  \abs{s} < \epsilon : \quad \M{f_{1}}\circ \mathsf{e}^{\mathbf{i} s D} \circ \M{f_{2}} =0.
	\]
\end{cor}
\begin{proof}
	By assumption, $K:= \mathsf{supp}(f_{2})$ is compact.	Since the support of $f_{1}$ is disjoint from $K$, the set $W:= M\setminus \mathsf{supp}(f_{1})$ is an open neighborhood of $K$. By \Cref{10.3.1}, there exists an $\epsilon >0$ such that
	\[
		\forall  \abs{s} < \epsilon, 
		\forall  v \in \mathrm{L}^2 (K;S),
		\quad
		\mathsf{supp}
		\left(
			\mathsf{e}^{\mathbf{i} s D} v
		\right)
		\subseteq W.
	\]
	For any $u\in\mathrm{L}^2 (M;S)$, we know that $\M{f_{2}}u$ is supported in $K$, so $\mathsf{e}^{\mathbf{i} s D} \M{f_{2}}u$ is supported in $W$. As $W= M\setminus \mathsf{supp}(f_{1})$, we hence get
	\[
	 	\M{f_{2}}\mathsf{e}^{\mathbf{i} s D} \M{f_{2}}u = 0
	\]
	for all $u\in\Gamma^\infty(M;S)$.
\end{proof}
\begin{lem}[Kasparov's lemma; {\cite[5.4.7]{HigRoe:KHom}} - without proof]\label{KasparovsLemma}
	Suppose $X$ is compact Hausdorff, $\nu:C(X)\to \B(\mathcal{H})$ a non\-/degenerate representation, and $T\in\B(\mathcal{H})$. If $\nu(f_{1}) T \nu(f_{2})$ is compact for every $f_{1},f_{2}\in C(X)$ with disjoint support, then $\comm{T}{\nu(f)}$ is compact for every $f\in C(X)$.
\end{lem}
\begin{prop}[special case of {\cite[Lemma~10.6.4]{HigRoe:KHom}}]\label{lem:CommutatorCpct}
	If $D$ a symmetric elliptic differential operator on a compact manifold $M$, $\chi$ a normalizing function, and $f\in C (M)$, then $\comm{\chi(D)}{\M{f}}$ is compact.
\end{prop}
\begin{proof}[Proof of \Cref{lem:CommutatorCpct}]
	Since $\mathsf{M}$ is a non\-/degenerate representation of $C(M)$ on $\mathrm{L}^2(M;S)$,
	\hyperref[KasparovsLemma]{Kasparov's lemma}  says that it suffices to show that, for all $f_{1}, f_{2}\in C(M)$ with disjoint supports, $\M{f_{1}} \chi(D) \M{f_{2}}$ is compact. Moreover, because of \Cref{lem:CpctPert}, we can actually show this for any normalizing function, and do not need to use the given $\chi$.
	
	So let us fix such $f_{1}, f_{2}$. By \Cref{lem:10.3.3}, there exists $\epsilon >0$ such that
	\[
		\forall  \abs{s} < \epsilon : \quad \M{f_{1}} \mathsf{e}^{2\pi\mathbf{i} s D} \M{f_{2}} =0.
	\]
	By \Cref{ChiExists}, we can take a normalizing function $\chi_{1}$ with $\mathsf{supp}(\hat{\chi}_{1})\subseteq (-\epsilon,  \epsilon)$. We then get by \Cref{lem:10.3.5} that, for all $\tilde{u}, \tilde{v}\in \Gamma^\infty (M;S)$ and ${g}(s):= \inner{\mathsf{e}^{2\pi\mathbf{i} sD}\tilde{u}}{\tilde{v}}$,
	\begin{align}\label{foo}
		\biginner{\chi_{1}(D)\tilde{u}}{\tilde{v}} 
		=
		\iinner{\hat{\chi}_{1}}{g}
		.
	\end{align}
	
	If we choose $\tilde{u}:= f_{2} \cdot u$ and $\tilde{v}:= \overline{f_{1}} \cdot v$ for $u, v\in \Gamma^\infty (M;S)$, then
	\[
	{g}(s)
	= \inner{\mathsf{e}^{2\pi\mathbf{i} sD}(f_{2} \cdot {u})}{\overline{f_{1}}\cdot{v}}
	=
	\inner{\M{f_{1}}\circ\mathsf{e}^{2\pi\mathbf{i} sD}\circ \M{f_{2}} ({u})}{{v}},
	\]
	so that $g(s) = 0$ for $\abs{s} < \epsilon$ by choice of $\epsilon$, and hence
	\[
	\iinner{\hat{\chi}_{1}}{{g}} = 0 \quad\text{ as }  \quad\mathsf{supp}(\hat{\chi}_{1})\subseteq (-\epsilon,  \epsilon).
	\]
	Thus, \Cref{foo} gives 
	$\biginner{\M{f_{1}}\chi_{1}(D)\M{f_{2}}{u}}{{v}}= 0$.	
	We conclude that the same even holds for $u, v\in \mathrm{L}^2 (M; S)$, so that we have proved $\M{f_{1}}\chi_{1}(D)\M{f_{2}}=0$.
\end{proof}

Finally, we can prove \Cref{thm:main}.
\begin{proof}[Proof of \Cref{thm:main}]\label{pf:thm-main}
	$F$ is self\-/adjoint by \Cref{FuncCalc-selfadj} of Functional Calculus because $\chi$ is real\-/valued. 
	Since $\chi$ is a normalizing function, $\chi^2 - 1 \in C_{0} (\RR)$, so \Cref{prop:ctsFuncCalcCpct} implies that $(\chi^2 - 1)(D) = F^2 -1$ is compact.
	\Cref{lem:CommutatorCpct} says $\comm{\chi(D)}{\M{f}}=\comm{F}{\M{f}}$ is compact for any $f\in C(M)$, so we have shown that the properties of a Fredholm module are satisfied.
	Lastly, if $\chi_{1}$ is another normalizing function, then by \Cref{lem:CpctPert} again, $\chi_{1}(D)$ differs from $\chi(D)$ only by a compact operator. This means that \mbox{$(\M{}, \mathcal{H}, \chi_{1}(D))$} is a compact perturbation of \mbox{$(\M{}, \mathcal{H}, F)$.} Therefore, they determine the same $\Kth$\=/homology class by \Cref{CompactPertImpliesHomotopic}.	
\end{proof}

%
%
\begin{rmk}
	There is an obvious extension of \Cref{thm:main} to even $\Kth$\-/homology: if $S$ is equipped with a smooth idempotent vector bundle automorphism $\gamma_S$ (that is, $S$ is  $\mathbb{Z}/ {2} \mathbb{Z}$\=/graded), then the map
	\[
		\gamma \colon  \Gamma_{c}^\infty (M;S)\to \Gamma_{c}^\infty (M;S),\quad
		\gamma u (p) := \gamma_S\bigl(u(p)\bigr),
	\]
	extends to a grading operator of $\mathcal{H} = \mathrm{L}^2 (M; S)$ with respect to which the left $C(M)$\=/action is even. If we further assume that $D$ is odd
	, then \Cref{lem:10.6.2} implies that $F$ is odd as well, so that the Fredholm module $(\M{}, \mathcal{H}, F)$ is actually graded. Again, the corresponding class in $\Kth^{0} ( C(M))$ only depends on $D$.
\end{rmk}

\begin{example}\label{ex:Dirac-on-mfd-falls-into-main-Thm}
	As discussed in \Cref{ex:Dirac-on-mfd-is-symmetric} and \Cref{ex:Dirac-is-elliptic}, the Dirac operator $\Dirac{M}$ of a spin$^c$ manifold $M$ is an unbounded, symmetric elliptic differential operator. Moreover, in case the dimension of the manifold is even, the spinor bundle is actually graded and $\Dirac{M}$ is an odd operator. Consequentially, if the manifold is compact, $\Dirac{M}$ determines a class $[\Dirac{M}]$ in the even or odd $\Kth$-homology of $C(M)$, depending on whether $\dim(M)$ is even or odd.
\end{example}

An interesting consequence of \Cref{thm:main} is that it gives rise to maps on $\Kth$-theory: if
\begin{align}
		\langle\,\cdot\,,\,\cdot\,\rangle\colon 
		\Kth_{j}(A)\times \Kth^{j}(A) \to \ZZ,
		\tag{$j=0, 1$}
\end{align}
denotes the index paring as defined in \cite[Prop.~8.7.1 and 8.7.2]{HigRoe:KHom}, then any symmetric elliptic differential operator $D$ on a smooth and compact manifold $M$ gives rise to a map
\[
	\Kth_{1} ( C(M)) \to \ZZ,
	\quad
	x\mapsto \langle x,[D]\rangle,
\]
by pairing a $\Kth$\-/theory class with the $\Kth$\-/homology class $[D]$ constructed above. If the vector bundle $S$ over $M$ which is underlying $D$ is graded and if $D$ is odd, then we get a map
\[
	\Kth_{0} ( C(M)) \to \ZZ,
	\quad
	x\mapsto \langle x,[D]\rangle.
\]

\section*{Appendix}

In order to prove \Cref{Fmolli}, we first need the following version for $\RR^n$:

\begin{lem}
	There exist operators $\tilde{F}_{t} : \mathrm{L}^2 (\RR^n)\to \mathrm{L}^2 (\RR^n)$ such that	
	\begin{enumerate}[label=\textup{(\alph*)}]
		\item\label[property]{tildeFt:norm} $\norm{\tilde{F}_{t}}\leq 1$,
		\item\label[property]{tildeFt:1} $\forall  u\in \mathrm{L}^2 (\RR^n): \, \lim_{t\to 0} \tilde{F}_{t} u = u$ in $\mathrm{L}^2 (\RR^n)$, 
		\item\label[property]{tildeFt:SmoothCpct} $\forall  u\in \mathrm{L}^2 (\RR^n):$ $\tilde{F}_{t} u$ is smooth,
		\item\label[property]{tildeFt:supp} if $u$ has compact support, then so does $\tilde{F}_{t} u$, and
		\item\label[property]{tildeFt:comm} for all $1\leq k\leq n$ and $f\in C^{\infty}(\RR^n)$ with bounded partial derivatives, the operator $\comm{f\cdot\frac{\partial}{\partial x_{k}}}{\tilde{F}_{t}}$ extends to a bounded operator whose norm is bounded independent of $t$. 
	\end{enumerate}
\end{lem}

\begin{proof}

Pick a smooth function $\phi\colon \RR^n \to \RR^+$  with compact support and $\int_{\RR^n} \phi \,\mathrm{d}\lambda  = 1$. Define $\phi_{t} (x) := t^{-n} \phi(\frac{x}{t})$, which has the same properties as $\phi$. Set $\tilde{F}_{t} u = \phi_{t} \ast u$ for $u\in \mathrm{L}^2 (\RR^n)$, that is:
\[
	\tilde{F}_{t} u (x)
	=	t^{-n} \int_{\RR^n} \phi \left(\frac{x-y}{t}\right) u(y)\,\mathrm{d}\lambda (y).
\]

By \cite[IV~9.4]{Werner:Ana}, we have
\[
	\norm{\tilde{F}_{t}}\leq \mathsf{sup}\{\norm{\phi_{t}}_{\mathrm{L}^1} \cdot\norm{u}_{2} \,:\, \norm{u}_{2} \leq 1\} = \norm{\phi_{t}}_{\mathrm{L}^1}  = 1,
\]
so \Cref{tildeFt:norm} holds. Moreover, \Cref{tildeFt:1} and 
\ref{tildeFt:SmoothCpct} follow from \cite[Satz~IV 9.5]{Werner:Ana} and \cite[Korollar~IV~9.7]{Werner:Ana} respectively. It is well known that
\begin{align*}\label{tildeFt:supp-2}
	\mathsf{supp}(\tilde{F}_{t} u) \subseteq \overline{\mathsf{supp}(\phi_{t}) + \mathsf{supp}(u)},
\end{align*}
so that \Cref{tildeFt:supp} follows from $\phi_{t}$ having compact support. It remains to check \Cref{tildeFt:comm}:

Using integration by parts and the fact that $\phi$ is compactly supported, we can compute for $u\in\mathrm{L}^2 (\RR^n)$
\[
	\comm{f\cdot\frac{\partial}{\partial x_{k}}}{\tilde{F}_{t}} u (x)
	=
	\int\limits_{y\in \RR^n}
	\left[
		\frac{1}{t^{n+1}}
		\left.
			{\frac{\partial \phi}{\partial x_{k}}}
		\right|_{\frac{x-y}{t}}\!\!
		\bigl(
			f(x)-f(y)
		\bigr)
		+
		\frac{1}{t^{n}}
		\phi
			\left(
				\frac{x-y}{t}
			\right)
		\left.
			{\frac{\partial f}{\partial x_{k}}}
		\right|_{y}
	\right]
	u(y) \,\mathrm{d}y .
\]
In other words, $\comm{f\cdot\frac{\partial}{\partial x_{k}}}{\tilde{F}_{t}}$ is an integral transform with kernel
\[
	k_{t} (x,y)
	=
	\frac{1}{t^{n+1}}
	\left.{\frac{\partial \phi}{\partial x_{k}}}\right|_{\frac{x-y}{t}}
	\bigl(
	f(x)-f(y)
	\bigr)
	+
	\frac{1}{t^{n}}
	\phi
	\left(
	\frac{x-y}{t}
	\right)
	\left.{\frac{\partial f}{\partial x_{k}}}\right|_{y}.
\]
As stated in \cite[Thm.~5.2]{HalSun:BddIntOps}, the so\-/called \textit{Schur's test} says that, if
\[
	\sup_{x\in \RR^n}{\norm{k_{t} (x,\cdot)}_{\mathrm{L}^1}} \leq \alpha
	\;\text{ and }\;
	\sup_{y\in \RR^n}{\norm{k_{t} (\cdot,y)}_{\mathrm{L}^1}} \leq \beta,
\]
then the integral transform extends to a bounded operator whose norm is bounded by $\sqrt{\alpha\beta}$. We claim that, if for all $1\leq j\leq n$ and $\mathsf{supp}(\phi) \subseteq [-a,a]$ we have
\begin{equation*}
	\supnorm{{\tfrac{\partial f}{\partial x_{j}}}} < C,
\end{equation*}
then $\alpha=\beta=C ( n a \norm{\frac{\partial \phi}{\partial x_{k}}}_{\mathrm{L}^1} + \norm{\phi}_{\mathrm{L}^1})$ do the trick.

For $x, y\in\RR^n$ such that $\norm{x-y} \leq a t$, repeated application of the Mean Value Theorem (see the proof of\cite[Thm.~5.3.10]{Trench:RA}, for example) gives
\[
	\abs{f(x) - f(y)}
	\leq 
	\sum_{j=1}^{n}
	a t
	\supnorm{{\frac{\partial f}{\partial x_{j}}}}
	\leq
	n a t C.
\]
For these $x,y$, we compute
\begin{align*}
	\abs{k_{t} (x,y)}
	&\leq
	\frac{1}{t^{n+1}}
	\abs{
		\left.{\frac{\partial \phi}{\partial x_{k}}}\right|_{\frac{x-y}{t}}
		\left(
			f(x)-f(y)
		\right)
	}
	+
	\frac{1}{t^{n}}
	\abs{
		\phi
		\left(
			\frac{x-y}{t}
		\right)
	\left.{\frac{\partial f}{\partial x_{k}}}\right|_{y}
	}
	\\
	&\leq
	\frac{1}{t^{n+1}}
	\abs{	
		\left.
			{\frac{\partial \phi}{\partial x_{k}}}
		\right|_{\frac{x-y}{t}}
	}
\cdot
	natC
	+
	\frac{1}{t^{n}}
	\abs{
		\phi
		\left(
			\frac{x-y}{t}
		\right)
	}
\cdot
	C\\
	&=
	\frac{1}{t^{n}}
	C
	\left[
		\abs{	
			\left.
			{\frac{\partial \phi}{\partial x_{k}}}
			\right|_{\frac{x-y}{t}}
		}
		na
		+
		\abs{
			\phi
			\left(
			\frac{x-y}{t}
			\right)
		}
	\right]
\end{align*}

For all other $x,y$, we have $k_{t} (x,y) = 0$ because $\phi$ is supported within $[-a,a]$. This means that the above calculation and a substitution shows that
\[	
	{\norm{k_{t} (x,\cdot)}_{\mathrm{L}^1}}, {\norm{k_{t} (\cdot,y)}_{\mathrm{L}^1}}
	\leq
	C \left( n a \norm{\frac{\partial \phi}{\partial x_{k}}}_{\mathrm{L}^1} + \norm{\phi}_{\mathrm{L}^1}\right),
\]
and we are done.
\end{proof}

\begin{lem}\label{lem:Cpct-as-nested-Gdelta}
	For $K$ a compact subset of a manifold $M$, we can write $K = \bigcap_{k=1}^{\infty} V_{k}$ for some open sets $V_{k+1}\subseteq V_{k} \subseteq M$.
\end{lem}

\begin{proof}
	For $K$ contained in some chart $(U,\varphi)$, we have
	\[
		\varphi(K) =
		\bigcap_{k=1}^{\infty} 
		\tilde{V}_{k},
		\quad
		\text{where }
		\tilde{V}_{k} :=
		\bigcup_{x\in \varphi(K)}
		B_{\frac{1}{k}} (x).
	\]
	If we then let $\frac{1}{N}$ be smaller than the distance of the compact set $\varphi(K)$ to the closed set \mbox{$\RR^n \setminus \varphi(U)$,} then for $k\geq N$ we have $\tilde{V}_{k} \subseteq \varphi(U)$, and hence
	\[
		K =
		\bigcap_{k=N}^{\infty} 
		V_{k},
		\quad
		\text{where }
		{V}_{k} :=
		\varphi^\mi (\tilde{V}_{k}).
	\]
	Now, for arbitrary $K$, take finitely many open sets $U_{1},\ldots, U_l$ which cover $K$ such that $\overline{K\cap U_{i}}$ is contained in a chart. From the above, we get (after re-indexing)
	\[
		K	= \overline{K\cap U_{1}} \cup \ldots \cup \overline{K\cap U_l}
			= \left( \bigcap_{k=1}^{\infty} V^1_{k} \right) \cup \ldots \cup \left(\bigcap_{k=1}^{\infty} V^l_{k}\right)
			\subset \bigcap_{k=1}^{\infty} \left( V^1_{k}  \cup \ldots \cup V^l_{k}\right).
	\]
	Since each family $\{V^i_{n}\}_{n}$ is nested, we also have
	\[
		\left( \bigcap_{k=1}^{\infty} V^1_{k} \right) \cup \ldots \cup \left(\bigcap_{k=1}^{\infty} V^l_{k}\right)	\supset \bigcap_{k=1}^{\infty} \left( V^1_{k}  \cup \ldots \cup V^l_{k}\right),
	\]
	and hence
		$K	= \bigcap_{k=1}^{\infty} \left( V^1_{k}  \cup \ldots \cup V^l_{k}\right).$
\end{proof}

\begin{lem}\label{po1-for-compact}
	For $K$ a compact subset of a manifold $M$ and $\{U_{i}\}_{i=1}^l$ an open cover of $K$ in $M$, there exist smooth compactly supported functions $\rho_{1},\ldots, \rho_l\colon  M \to [0,1]$ such that $\mathsf{supp}(\rho_{i})\subseteq U_{i}$ and $\sum_{i=1}^l \rho_{i}(p) = 1$ for all $p\in K$.
\end{lem}

\begin{proof}
	With $U_{0} := M\setminus K$, take a partition of unity $\{\rho_{i}\}_{i=0}^l$ of $M$ subordinate to the cover $\{U_{i}\}_{i=0}^l$. Since $M = U_{0} \cup U_{1} \cup \ldots \cup U_l$, we get from \cite[Lemma~1.4.8]{Conlon:DiffMa} that there exists an open cover $\{V_{i}\}_{i=0}^{l}$ of $M$ with $\overline{V}_{i} \subseteq U_{i}$ for all~$0\leq i \leq l$. Now since $\overline{K\cap V_{i}} \subseteq \overline{V}_{i} \subseteq U_{i}$ for $i\neq 0$, and $\overline{K\cap V_{i}}\subseteq K$ is compact, we know by \cite[Prop.~4.31]{Folland:RA} that there exists a precompact open set $W_{i}$ such that
	\[
		\overline{K\cap V_{i}} \subseteq W_{i} \subseteq \overline{W_{i}} \subseteq U_{i}.
	\]
	Note that the collection of $W_{i}$'s covers all of $K$, so that we can take a smooth partition of unity $\{\rho_{i}\}_{i=0}^l$ of $M$ which is subordinate to $\{M\setminus K\}\cup\{W_{i}\}_{i=1}^l$. In particular, since $W_{i}$ is precompact and $\mathsf{supp}(\rho_{i})\subseteq W_{i}$ for $i>0$, we know that those $\rho$'s have compact support. Moreover, it follows from $\mathsf{supp}(\rho_{0} )\subseteq M\setminus K$ that for $p\in K$
\begin{equation*}
		1 = \sum_{i=0}^l \rho_{i}(p) = \sum_{i=1}^l \rho_{i}(p).\qedhere
\end{equation*}
\end{proof}

\begin{lem*}[\Cref{Fmolli}]	
	For $M$ and $S$ as specified at the beginning of \Cref{sec:ellops}, and any $K\subseteq M$ compact, there exist operators $F_{t}\colon  \mathrm{L}^2 (K; S) \to \mathrm{L}^2 (M; S)$ for sufficiently small $\epsilon > t > 0$ which satisfy
	\begin{enumerate}[label=\textup{(\arabic*)}]
		\item $\norm{F_{t}}\leq C$ for some constant $C$ and all $t$,
		\item $\forall  u\in \mathrm{L}^2 (K; S): \, \lim_{t\to 0} F_{t} u = u$ in $\mathrm{L}^2 (M; S)$,
		\item $\forall  u\in \mathrm{L}^2 (K; S):$ $F_{t} u$ is smooth with compact support, and
		\item for any differential operator $D$ on $M$, $\comm{D}{F_{t}}$ extends to a bounded operator \linebreak[4]\mbox{$\mathrm{L}^2 (K; S) \to \mathrm{L}^2 (M; S)$,} and its norm is bounded independent of $t$.
	\end{enumerate}
\end{lem*}


\noindent\textit{Proof of \Cref{Fmolli}.}\label{pf:Fmolli}

Take an atlas $\mathcal{A}$ of $M$ whose charts are small enough to also allow smooth trivializations of $S$ which are isometries on the fibres,
	\[
		\begin{tikzcd}
			\hphantom{M \supseteq}S_{| {U}}
			\arrow[d, "\pi", shift left=2.2ex]
			   \arrow[rr, "\approx"]{} &&  {U}\times \mathbb{C}^k \hphantom{\subseteq \RR^n } 
			\arrow[dll, "\mathrm{pr}_{1}"']{}\\
			M \supseteq  {U}   	\arrow[rr, "\approx"]{}	&&	\RR^n 
		\end{tikzcd}
	\]
	
Let $\{(U_{i}, \varphi_{i})\}_{i=1}^l$ be finitely many of those charts which cover the compact set $K$, and let $\{\rho_{i}\}_{i=1}^l$ be as in \Cref{po1-for-compact}. For our trivializations, we will write
\[
	\Psi_{i}\colon S_{| {U}_{i}} \stackrel{\approx}{\longrightarrow}  {U}_{i} \times \mathbb{C}^k, 
	\quad
	\psi_{i} := \mathrm{pr}_{2} \circ \Psi_{i}.
\]

Moreover, let $f_{i}, g_{i}=\frac{1}{f_{i}} : \RR^n \to (0,\infty)$ be such that for all $h\in C_{c}^{\infty} (\RR^n)$ and $E\subseteq {U}_{i}$ Borel, we have
\begin{align*}
	\int\limits_E h\circ\varphi_{i} \,\mathrm{d}\mu  =	\int\limits_{\varphi_{i} (E)} h \cdot f_{i} \,\mathrm{d}\lambda 
	\quad\;\text{ and }\;\quad
	\int\limits_E (h\cdot g_{i}) \circ \varphi_{i} \,\mathrm{d}\mu 	= \int\limits_{\varphi_{i} (E)} h \,\mathrm{d}\lambda .
\end{align*}
We have assumed in \Cref{rmk:standAssu-L} that $\supnorm{f_{i}},\supnorm{g_{i}}\leq L$ for some number $L$. In particular, we have for \mbox{$v\in \bigoplus_{1}^k \mathrm{L}_{c}^2 (\RR^n)$} and any $1\leq i\leq l$, 
\begin{align}\label{mitf}
	\int\limits_{U_{i}}
	\norm{
		v
		\circ
		\varphi_{i} (p)
	}_{\mathbb{C}^k}^2
	\,\mathrm{d}\mu 
	&=\int\limits_{\RR^n} 
	\norm{v
		(x)}
	_{\mathbb{C}^k}^2\cdot {f_{i}}(x) \,\mathrm{d}\lambda 
	\leq
	\norm{v
	}^2_{2}	\cdot L
	.
\end{align}
For $u\in L^2 (U_{i};S)$, since $\psi_{i}$ is isometric we get
\begin{align}
	\int\limits_{\RR^n} 
	\norm{\bigl(
		\psi_{i} \circ u \circ \varphi_{i}^\mi
		\bigr)
		(x)}
	_{\mathbb{C}^k}^2 \,\mathrm{d}\lambda 
	&=
	\int\limits_{\RR^n} 
		\norm{
			\bigl(
			u \circ \varphi_{i}^\mi
			\bigr)
			(x)
		}
		_{S_{\varphi_{i}^{-1} (x)}}^2 \,\mathrm{d}\lambda 
		\notag
	\\
	&=
	\int\limits_{U_{i}} 
	\norm{u (p)}
	_{S_{p}}^2 g_{i} \bigl(\varphi_{i} (p)\bigr)  \,\mathrm{d}\mu 
	\leq
	\norm{u}^2_{2} \cdot L.
	\label{psiuintermsofu}
\end{align}

For $1\leq i \leq l$, we define
\[
	\begin{tikzcd}
			&
			&
			&
			\bigoplus_{1}^k \mathrm{L}^2 (\RR^n)\arrow[d, phantom, "\text{\rotatebox[origin=c]{90}{$\subseteq$}}"]
			&
			\mathrm{L}^2 (U_{i};S)\arrow[d, phantom, "\text{\rotatebox[origin=c]{90}{$\subseteq$}}"]
			\\[-10pt]
		F_{t}^i\colon 
		&
		\mathrm{L}_{c}^2 (U_{i};S)
		\arrow[r]
		&
		\bigoplus_{1}^k \mathrm{L}_{c}^2 (\RR^n)
		\arrow[r]
		&
		\bigoplus_{1}^k C_{c}^\infty (\RR^n)
		\arrow[r]
		&
		\Gamma_{c}^\infty (U_{i};S)
		\\[-20pt]
		&
		u
		\arrow[r, mapsto]
		&
		\psi_{i} \circ u \circ \varphi_{i}^\mi,
		&
		v \arrow[r, mapsto]
		&
		\Psi_{i}^\mi (\,\cdot\,, v\circ \varphi_{i})
		\\[-20pt]
		&
		&
		\oplus_{j = 1}^k w_{j}
		\arrow[r, mapsto]
		&
		\oplus_{j = 1}^k \tilde{F}_{t} w_{j}
		&
	\end{tikzcd}
\]

Notice that, indeed, $F_{t}^i$ takes values in $\Gamma_{c}^\infty (U_{i};S)$: since $\varphi_{i}$ is a diffeomorphism, $\psi_{i} \circ u \circ \varphi_{i}^\mi$ is compactly supported when $u$ is, and in particular, all of its component functions are compactly supported. By \Cref{tildeFt:SmoothCpct} of $\tilde{F}_{t}$, $\tilde{F}_{t} w_{j}$ is smooth, and by \Cref{tildeFt:supp}, it is compactly supported when $w_{j}$ is. Since $\varphi_{i}$ and $\Psi_{i}$ are smooth, so is $\Psi_{i}^\mi (\,\cdot\,, v\circ \varphi_{i})$ for smooth $v$, and again, since $\varphi_{i}$ is a diffeomorphism, we conclude that $\Psi_{i}^\mi (\,\cdot\,, v\circ \varphi_{i})$ has compact support for compactly supported $v$.

Now let
\[
	F_{t}\colon  \mathrm{L}^2 (K;S) \to \mathrm{L}^2 (M;S),\quad F_{t} u := \sum_{i=1}^l F_{t}^i ( \rho_{i} \cdot u).
\]

By the above explanation, $F_{t}$ actually takes values in $\Gamma^\infty_{c} (M;S)$ because the $\rho_{i}$ are compactly supported. Hence, $F_{t}$ satisfies \Cref{Ft:SmoothCpct}, and we need to check the other properties. By abuse of notation, we will write $\tilde{F}_{t}$ for the operator $\oplus_{1}^k \tilde{F}_{t}$.

\begin{description}[leftmargin=*]
	\item[ad Property \emph{\ref{Ft:norm}}] For $u\in \mathrm{L}^2 (U_{i};S)$, we compute		
	\begin{align}
		\norm{F_{t}^i u}_{2}^2 
		&=
		\int\limits_{U_{i}}
			\norm{
				\Psi_{i}^\mi 
				\Bigl(
					p,  
					\tilde{F}_{t}
					\bigl(
						\psi_{i} \circ u \circ \varphi_{i}^\mi
					\bigr)
					\circ
					\varphi_{i}(p)
				\Bigr)
			}_{S_{p}}^2
		\,\mathrm{d}\mu 
		\notag\\
		&=
		\int\limits_{U_{i}}
		\norm{
			\tilde{F}_{t}
				\bigl(
					\psi_{i} \circ u \circ \varphi_{i}^\mi
				\bigr)
				\circ
				\varphi_{i} (p)
		}_{\mathbb{C}^k}^2
		\,\mathrm{d}\mu 
		\tag*{as $\Psi_{i}(p,\,\cdot\,)$ is an isometry}
		\\
		&
		\leq
		\norm{\tilde{F}_{t}
			\bigl(
			\psi_{i} \circ u \circ \varphi_{i}^\mi
			\bigr)
		}^2_{2}	\cdot L
		\tag*{by Eq.~(\ref{mitf})}\\
		&\leq
		\norm{\psi_{i} \circ u \circ \varphi_{i}^\mi}^2_{2} \cdot L
		\leq 
				\norm{u}^2_{2} \cdot L^2
		\tag*{since $\norm{\tilde{F}_{t}}\leq 1$ and by Eq.~(\ref{psiuintermsofu}).}
	\end{align}
	Hence $\norm{F_{t}^i} 	\leq L$,	so that
	\begin{align*}
		\norm{F_{t}}
		=
		\sup_{\norm{u}_{2} \leq 1}{\norm{\sum_{i=1}^l F_{t}^i (\rho_{i} u)}_{2}}
		\leq
		\sum_{i=1}^l \sup_{\norm{u}_{2} \leq 1}{\norm{F_{t}^i (\rho_{i} u)}_{2}}
		\leq
		\sum_{i=1}^l \norm{F_{t}^i}
		\leq
		l\cdot L =: C.
	\end{align*}	
		
	\item[ad Property \emph{\ref{Ft:1}}] As $\Psi_{i} (p,\,\cdot\,)$ is an isometry, we have for $u$ in $\mathrm{L}^2 (U_{i};S)$, $p\in U_{i}$:
	\begin{align*}
	\norm{\bigl(F_{t}^i u -u\bigr) (p)}_{S_{p}}
	&=
	\norm{
	\Psi_{i}^\mi 
	\Bigl(p,  
		\tilde{F}_{t}
		\bigl(
			\psi_{i} \circ u \circ \varphi_{i}^\mi
		\bigr)
		\circ
		\varphi_{i}(p)
	\Bigr)
		-
		u(p)
	}_{S_{p}}\\
	&=
	\norm{
		\tilde{F}_{t}
			\bigl(
			\psi_{i} \circ u \circ \varphi_{i}^\mi
			\bigr)
			\circ
			\varphi_{i}(p)
		-
		\bigl(\psi_{i} \circ u \circ \varphi_{i}^\mi\bigr)(\varphi_{i}(p))
	}_{\mathbb{C}^k}.
	\end{align*}
	
	By \Cref{mitf}, it therefore follows that
	\begin{align*}
		\norm{F_{t}^i u -u}_{2}^2 
		&\leq
		\norm{
			\tilde{F}_{t}
			\bigl(
			\psi_{i} \circ u \circ \varphi_{i}^\mi
			\bigr)
			-
			\bigl(\psi_{i} \circ u \circ \varphi_{i}^\mi\bigr)
		}^2_{2}	\cdot L,
	\end{align*}
	so that \Cref{tildeFt:1} of $\tilde{F}_{t}$ implies that $\lim_{t\to 0}F_{t}^i u = u$ in $\mathrm{L}^2$\=/norm. Therefore, for arbitrary $u\in \mathrm{L}^2 (K;S)$,
	\[
		\norm{ F_{t} u - u}_{2}
		=
		\norm{ \sum_{i=1}^l (F_{t}^i (\rho_{i} u) - \rho_{i} u)}_{2}
		\leq
		\sum_{i=1}^l 
		\norm{ F_{t}^i (\rho_{i} u) - \rho_{i} u}_{2}
		\longrightarrow 0.
	\]

	\item[ad Property \emph{\ref{Ft:bdd}}] Suppose $D$ is a differential operator acting on the sections of $S$. Since $\rho_{i} u$ is supported in $U_{i}$ for $u\in  \Gamma^\infty (M;S)$, we know that $D(\rho_{i} u) \in \Gamma^\infty (U_{i};S)$ by \hyperref[DiffOpPreservesSupp]{Property~a)} of differential operators. Therefore, $F_{t}^i \bigl(D (\rho_{i} u)\bigr)$ also makes sense, and we can write
	\begin{align*}
		\comm{F_{t}}{D}u
		&=
		\sum_{i=1}^l F_{t}^i \bigl( \rho_{i} Du\bigr) - D \bigl( F_{t}^i (\rho_{i} u)\bigr)
		\\&=
		\sum_{i=1}^l F_{t}^i \bigl( \rho_{i} Du\bigr) - F_{t}^i \bigl(D (\rho_{i} u)\bigr) + F_{t}^i \bigl(D (\rho_{i} u)\bigr)  - D \bigl( F_{t}^i (\rho_{i} u)\bigr)\\
		&=
		\sum_{i=1}^l F_{t}^i \comm{\M{\rho_{i}}}{D}u + \comm{F_{t}^i}{D} (\rho_{i}u),
	\end{align*}
	that is,
	\begin{align}
		\comm{F_{t}}{D} = \sum_{i=1}^l F_{t}^i \comm{\M{\rho_{i}}}{D} + \comm{F_{t}^i}{D}\M{\rho_{i}}.
	\end{align}
	
	In order to check that $\comm{F_{t}}{D}$ extends to an operator that is bounded independent of $t$, we will show that $F_{t}^i \comm{\M{\rho_{i}}}{D}$ and $\comm{F_{t}^i}{D}\M{\rho_{i}}$ do. 
	As was shown in \Cref{commDMisbdd}, $\comm{\M{\rho_{i}}}{D}$ is a bounded operator on $\mathrm{L}^2 (K;S)$,
	and since $F_{t}^i$ is bounded independent of $t$ (namely by $L$, as was shown above), so is $F_{t}^i \comm{\M{\rho_{i}}}{D}$.
	It remains to show that $u \mapsto \comm{F_{t}^i}{D} (\rho_{i} u)$ for a fixed but arbitrary \mbox{$1\leq i \leq l$} is bounded independent of $t$. It suffices to consider those $D$ that (locally)  look like only one of the summands in \Cref{DiffOpLook}. First, recall that for a $\mathbb{C}^k$-vector valued function $w$ on $\RR^n$, we have
	\[
		\psi_{i}
		\circ 
		\Psi_{i}^\mi
			\bigl(\,\cdot\,,
			w(\,\cdot\,)
			\bigr)
		=
		w(\,\cdot\,),
	\]
	so for $u_{i} := \rho_{i} u$ and $p\in U_{i}$, we compute
	\begin{align*}
		\left(F_{t}^i D u_{i} \right)(p)
		&=
		\Psi_{i}^\mi 
		\Bigl(
			p,  
			\tilde{F}_{t}
			\bigl(
				\psi_{i} \circ (Du_{i}) \circ \varphi_{i}^\mi
			\bigr)
			\circ
			\varphi_{i}(p)
		\Bigr)\\
		&=
		\Psi_{i}^\mi 
		\biggl(
			p,  
			\tilde{F}_{t}
			\Bigl(
				(A\circ\varphi_{i}^\mi) \cdot {\partial_{j} \bigl(\psi_{i} \circ u_{i}\circ \varphi_{i}^\mi\bigr)}
			\Bigr)
			\circ
			\varphi_{i}(p)
		\biggr)
	\end{align*}
	and
		\begin{align*}
		\left( D F_{t}^i u_{i} \right)(p)
		&=
		\Psi_{i}^\mi 
		\Bigl(
			p,  
			A(p) 
			\partial_{j} 
			\bigl(
				\psi_{i} \circ (F_{t}^i u_{i}) \circ \varphi_{i}^\mi
			\bigr)_{|\varphi_{i}(p)}
		\Bigr)\\
		&=
		\Psi_{i}^\mi 
			\biggl(
				p,  
				A(p) 
				\partial_{j} 
				\bigl(
					\tilde{F}_{t} (\psi_{i} \circ u_{i}\circ \varphi_{i}^\mi)
				\bigr)_{|\varphi_{i}(p)}.
			\biggr)
	\end{align*}
	If we write $\tilde{A}:= A\circ\varphi_{i}^\mi$ and $v_{i}:=\psi_{i} \circ u_{i}\circ \varphi_{i}^\mi$, then this means
	\begin{align*}
	\norm{\comm{F_{t}^i}{ D} u_{i} (p)}_{S_{p}}
	=
	\norm{
	\left(
		\tilde{F}_{t}
		\Bigl(
			\tilde{A} \cdot {\partial_{j} v_{i}}
		\Bigr)
		-
		\tilde{A}\cdot
		\partial_{j} 
		\bigl(
		\tilde{F}_{t} v_{i}
		\bigr)
	\right) \circ \varphi_{i}(p)
	}_{\mathbb{C}^k}.
\end{align*}
	Hence by \Cref{mitf},
	\begin{align*}
		\norm{\comm{F_{t}^i}{ D} u_{i}}_{2}^2
		\leq
		\norm{
			\comm{\tilde{F}_{t}}{\tilde{A} \cdot \partial_{j} }v_{i}
		}^2_{2}
		\cdot L.
	\end{align*}
	Note that $v_{i}$ is supported in the compact set $\kappa:=\varphi_{i} \bigl(\mathsf{supp}(\rho_{i})\bigr)\subseteq \RR^n$. Therefore, \Cref{tildeFt:comm} of $\tilde{F}_{t}$ implies that \mbox{$\comm{\tilde{F}_{t}}{\tilde{A} \cdot \partial_{j} }$} extends to an operator on $\oplus_{1}^{k} L^2(\kappa)$ which is bounded by, say, $c$ independent of $t$. Since moreover
	\begin{align*}
		\norm{v_{i}}^2_{2}
		=
		\norm{\psi_{i} \circ u_{i}\circ \varphi_{i}^\mi}^2_{2}
		\stackrel{\text{(\ref{psiuintermsofu})}}{\leq}
		\norm{u_{i}}^2_{2} \cdot L,
	\end{align*}
	we conclude
	\[
		\norm{\comm{F_{t}^i}{ D} u_{i}}^2_{2}
		\leq
		c\cdot \norm{v_{i}}^2_{2} \cdot L
		\leq 
		c\cdot \norm{u_{i}}_{2}^2\cdot L^2
		\leq
		c\cdot \norm{u}_{2}^2\cdot L^2.
	\]
	As neither $L$ nor $c$ depend on $t$, and this holds true for every $1\leq i\leq l$, we are done.\qed
\end{description}

\begin{rmk}
	Note that we do not mind our construction in the proof of \Cref{Fmolli} to be highly dependent on our choice of atlas and partition of unity.
\end{rmk}

\sloppy

\end{document}